\documentclass[11pt,letterpaper]{amsart}%
\usepackage{amsfonts}
\usepackage{geometry}
\usepackage{amsmath}
\usepackage{amssymb}
\usepackage{graphicx}
\usepackage{version}%
\providecommand{\U}[1]{\protect\rule{.1in}{.1in}}
\newtheorem{theorem}{Theorem}
\theoremstyle{plain}

\newtheorem{definition}{Definition}

\newtheorem{lemma}{Lemma}

\numberwithin{equation}{section}
\begin{document}
\title[Trispectrum for non-Gaussian field on the 2D-plane]{Trispectrum and higher order spectra for non-Gaussian homogenous and isotropic
field on the 2D-plane}
\author{Gy\"{o}rgy Terdik}
\address{Faculty of Informatics, University of Debrecen, Hungary, }
\email{Terdik.Gyorgy$@$inf.unideb.hu}
\date{April 8, 2016}
\subjclass[2000]{Primary 60G60; Secondary 60G10}
\keywords{Homogenous fields, Isotropic fields, Non-Gaussian field, Tricovariances,
Trispectrum, Higher order cumulants, Higher order spectra, 2D, Isotropic
process on circle.}

\begin{abstract}
In this paper we study the non-Gaussian homogenous and isotropic field on the
plane in frequency domain. The trispectrum and higher order spectra of such a
field are described in terms of Bessel functions. Some particular integrals of
Bessel functions are considered as well.

\end{abstract}
\maketitle

\section{Introduction}

In several fields of sciences like geophysics, astrophysics, climatology etc.
we come across observations which are non-Gaussian . A Gaussian process is
characterized by its first two moments, namely, mean, variance and
autocorrelations (or equivalently second order spectrum). There are several
clearly different processes having identical second-order properties
\cite{SRao97}, \cite{Igloi_Terd-Gyor}, \cite{diggle2013statistical} but their
distributions is not Gaussian and they are clearly different, therefore it is
necessary to study higher order structures. Although second order properties
of Gaussian fields are well established, see \cite{Yaglom-book-87},
\cite{Yadrenko1983}, \cite{adler2010geometr}, \cite{Brill-book-01},
\cite{Rosenblatt_Lin}, \cite{Rosenblatt1985}, \cite{Priest6},
\cite{leonenko1989statistical}, \cite{baxendale1986isotropic},
\cite{leonenko2012spectral}, \cite{moklyachuk2007prediction},
\cite{Moklyachuk2008}. Characterization of non-Gaussian fields which require
study of higher order moments (or equivalently higher order spectra) are not
well known. The isotropy of stochastic phenomenons in two dimensions has been
established in several applications, \cite{nozawa1997spectral},
\cite{ogura1996scattering}. Only recently has been made quite a number of
steps toward the statistical investigation of non-Gaussian isotropic fields,
mainly for understanding the Cosmic Microwave Background (CMB) anisotropies
\cite{Okamoto2002}, \cite{Adshead2012}.

In this paper our object is to continue the frequency domain investigations
published in our paper \cite{Terdik2014Publi}, where additional to the
covarinace and the spectrum the connection between the bicovariance and the
bispectrum has been studied in details. Here the trispectrum and all higher
order spectra of such fields are described in terms of Bessel functions. The
connection between the tricovariance and the trispectrum is similar to the one
between the bicovariance and the bispectrum It was necessary to prove some
particular integrals of Bessel functions, given in terms of sides and angles
of a multilateral.

\subsection{Isotropy}

A homogenous real valued stochastic field $X\left(  \underline{x}\right)  $,
$\underline{x}\in\mathbb{R}^{2}$, which is continuous (in mean square sense),
has the spectral representation
\[
X\left(  \underline{x}\right)  =\int_{\mathbb{R}^{2}}e^{i\underline{x}%
\cdot\underline{\omega}}Z\left(  d\underline{\omega}\right)  ,\quad
\underline{\omega},\underline{x}\in\mathbb{R}^{2},
\]
with $EX\left(  \underline{x}\right)  =0$ and the orthogonal spectral measure
$Z\left(  d\underline{\omega}\right)  $ has $E\left\vert Z\left(
d\underline{\omega}\right)  \right\vert ^{2}=F_{0}\left(  d\underline{\omega
}\right)  $. Homogeneity is defined in strict sense i.e. all the finite
dimensional distributions of $X\left(  \underline{x}\right)  $ are translation
invariant, see \cite{Yaglom-book-87} for details. We can rewrite $X\left(
\underline{x}\right)  $ in terms of polar coordinates
\begin{equation}
X\left(  r,\varphi\right)  =\int_{0}^{\infty}\int_{0}^{2\pi}e^{i\rho
r\cos\left(  \varphi-\eta\right)  }Z\left(  \rho d\rho d\eta\right) \nonumber
\end{equation}
where $\underline{x}=\left(  r,\varphi\right)  $, $\underline{\omega}=\left(
\rho,\eta\right)  $ are polar coordinates, $r=\left\vert \underline
{x}\right\vert =\sqrt{x_{1}^{2}+x_{2}^{2}}$, $\rho=\left\vert \underline
{\omega}\right\vert $, $\underline{x}\cdot\underline{\omega}=r\rho\cos\left(
\varphi-\eta\right)  $. This representation provides an isotropic field if
$F_{0}\left(  d\underline{\omega}\right)  $ is isotropic, i.e. $F_{0}\left(
d\underline{\omega}\right)  =E\left\vert Z\left(  d\underline{\omega}\right)
\right\vert ^{2}=E\left\vert Z\left(  \rho d\rho d\eta\right)  \right\vert
^{2}=F\left(  \rho d\rho\right)  d\eta$. The isotropy usually is defined
through the invariance of the covariance structure under rotations. A rotation
$g\in SO\left(  2\right)  $ is characterized by an angle $\gamma$. We consider
rotations about the origin of the coordinate system. If $\underline{x}%
\in\mathbb{R}^{2}$ is given in polar coordinates $\underline{x}=\left(
r,\varphi\right)  $, then $g\underline{x}=\left(  r,\varphi-\gamma\right)  $,
and as usual the operator $\Lambda\left(  g\right)  $ acts on functions
$f\left(  r,\varphi\right)  $, such that $\Lambda\left(  g\right)  f\left(
r,\varphi\right)  =f\left(  g^{-1}\left(  r,\varphi\right)  \right)  =f\left(
r,\varphi+\gamma\right)  $.

The invariance of the covariance function is satisfactory for Gaussian cases
but for non-Gaussian fields we need invariance of higher order cumulants as well.

\begin{definition}
A homogenous stochastic field $X\left(  \underline{x}\right)  $ is strictly
isotropic if all finite dimensional distributions of $X\left(  \underline
{x}\right)  $ are invariant under rotations.
\end{definition}

As far as the homogenous field $X\left(  \underline{x}\right)  $ is Gaussian
the isotropy of the spectral measure $F_{0}\left(  d\underline{\omega}\right)
$, i.e. in polar coordinates $F_{0}\left(  d\underline{\omega}\right)
=F\left(  \rho d\rho\right)  d\eta$, implies
\[
\operatorname*{Cov}\left(  \Lambda\left(  g\right)  X\left(  \underline{x}%
_{1}\right)  ,\Lambda\left(  g\right)  X\left(  \underline{x}_{2}\right)
\right)  =\operatorname*{Cov}\left(  X\left(  \underline{x}_{1}\right)
,X\left(  \underline{x}_{2}\right)  \right)  ,
\]
for each $\underline{x}_{1}$, $\underline{x}_{2}$ and for every $g\in
SO\left(  2\right)  $. It will be convenient for us later if we assume the
existence of all moments of the field $X\left(  \underline{x}\right)  $, in
this way from the isotropy follows that all higher order moments and cumulants
are also invariant under rotations.

Let us consider a homogenous and isotropic stochastic field $X\left(
\underline{x}\right)  =X\left(  r,\varphi\right)  $, ($r>0$, $\varphi
\in\left[  0,2\pi\right)  $) on the plane and put it into spectral
representation, see \cite{Terdik2014Publi}
\begin{equation}
X\left(  r,\varphi\right)  =\sum_{\ell=-\infty}^{\infty}e^{i\ell\varphi}%
\int_{0}^{\infty}J_{\ell}\left(  \rho r\right)  Z_{\ell}\left(  \rho
d\rho\right)  , \label{SeriesRepr_X}%
\end{equation}
where $J_{\ell}$ denotes the Bessel function of the first kind, and
\begin{equation}
Z_{\ell}\left(  \rho d\rho\right)  =\int_{0}^{2\pi}i^{\ell}e^{-i\ell\eta
\ }Z\left(  \rho d\rho d\eta\right)  . \label{measure_Z_l}%
\end{equation}
$Z_{\ell}$ is an array of measures, orthogonal to each other%
\[
\operatorname*{Cov}\left(  Z_{\ell_{1}}\left(  \rho_{1}d\rho_{1}\right)
,Z_{\ell_{2}}\left(  \rho_{2}d\rho_{2}\right)  \right)  =\delta_{\ell_{1}%
-\ell_{2}}F\left(  \rho d\rho\right)  ,
\]
where $\delta_{\ell}$ denotes the Kronecker-$\delta$. Note that the spectral
measure $F\left(  \rho d\rho\right)  $ of the stochastic spectral measure
$Z_{\ell}\left(  \rho d\rho\right)  $ does not depend on $\ell$. In
representation (\ref{SeriesRepr_X}) $e^{i\ell\varphi}$ plays a role of
spherical harmonics of degree $\ell$ with complex values on the plane. It
follows that an isotropic random field $X\left(  \underline{x}\right)  $ can
be decomposed into a countable number of mutually uncorrelated spectral
measures with a one dimensional parameter.

The isotropy of $X\left(  r,\varphi\right)  $ implies that the distribution of
$X\left(  r,\varphi\right)  $ does not change under rotations $g\in SO\left(
2\right)  $
\begin{align*}
\Lambda\left(  g\right)  X\left(  r,\varphi\right)   &  =\sum_{\ell=-\infty
}^{\infty}e^{i\ell\varphi}\int_{0}^{\infty}J_{\ell}\left(  \rho r\right)
e^{i\ell\gamma}Z_{\ell}\left(  \rho d\rho\right) \\
&  =\sum_{\ell=-\infty}^{\infty}e^{i\ell\varphi}\int_{0}^{\infty}J_{\ell
}\left(  \rho r\right)  Z_{\ell}\left(  \rho d\rho\right)  ,
\end{align*}
hence the distribution of $Z_{\ell}\left(  \rho d\rho\right)  $ and
$e^{i\ell\gamma}Z_{\ell}\left(  \rho d\rho\right)  $ should be the same.
Therefore under isotropy assumption we have
\[
\operatorname*{Cum}\left(  Z_{\ell_{1}}\left(  \rho_{1}d\rho_{1}\right)
,Z_{\ell_{2}}\left(  \rho_{2}d\rho_{2}\right)  \right)  =e^{i\left(  \ell
_{1}+\ell_{2}\right)  \gamma}\operatorname*{Cum}\left(  Z_{\ell_{1}}\left(
\rho_{1}d\rho_{1}\right)  ,Z_{\ell_{2}}\left(  \rho_{2}d\rho_{2}\right)
\right)  ,
\]
for each $\gamma$, hence either $\ell_{1}+\ell_{2}=0$, or $\operatorname*{Cum}%
\left(  Z_{\ell_{1}}\left(  \rho_{1}d\rho_{1}\right)  ,Z_{\ell_{2}}\left(
\rho_{2}d\rho_{2}\right)  \right)  =0$. In general, under isotropy assumption
we have
\[
\operatorname*{Cum}\left(  Z_{\ell_{1}}\left(  \rho_{1}d\rho_{1}\right)
,\ldots,Z_{\ell_{p}}\left(  \rho_{p}d\rho_{p}\right)  \right)  =e^{i\left(
\ell_{1}+\ell_{2}\cdots+\ell_{p}\right)  \gamma}\operatorname*{Cum}\left(
Z_{\ell_{1}}\left(  \rho_{1}d\rho_{1}\right)  ,\ldots,Z_{\ell_{p}}\left(
\rho_{p}d\rho_{p}\right)  \right)  ,
\]
that is either $\ell_{1}+\ell_{2}\cdots+\ell_{p}=0$, or $\operatorname*{Cum}%
\left(  Z_{\ell_{1}}\left(  \rho_{1}d\rho_{1}\right)  ,Z_{\ell_{2}}\left(
\rho_{2}d\rho_{2}\right)  ,\ldots,Z_{\ell_{p}}\left(  \rho_{p}d\rho
_{p}\right)  \right)  =0$, should be fulfilled. In turn, if this assumption
fulfils for each $p$, then the cumulants \newline$\operatorname*{Cum}\left(
Z_{\ell_{1}}\left(  \rho_{1}d\rho_{1}\right)  ,\ldots,Z_{\ell_{p}}\left(
\rho_{p}d\rho_{p}\right)  \right)  $ are invariant under rotations and if in
addition the distributions of $X\left(  r,\varphi\right)  $ are determined by
the moments then the field is isotropic.

\subsection{Spectrum and Bispectrum}

It is well known from theory of Gaussian fields that
\[
\operatorname*{Cov}\left(  X\left(  \underline{x}\right)  ,X\left(
\underline{y}\right)  \right)  =\int_{0}^{\infty}J_{0}\left(  \rho r\right)
F\left(  \rho d\rho\right)  ,
\]
where $r=\left\vert \underline{x}-\underline{y}\right\vert $,
\cite{Yadrenko1983}, \cite{Yaglom-book-87}, \cite{Brilling74}. The covariances
and the spectral measure uniquely define each other since the Hankel transform
gives the inverse. For absolutely continuos spectral measure we have $F\left(
\rho d\rho\right)  =\sigma^{2}\left\vert A\left(  \rho\right)  \right\vert
^{2}\rho d\rho$ and therefore
\begin{align*}
\mathcal{C}_{2}\left(  r\right)   &  =\int_{0}^{\infty}J_{0}\left(  \rho
r\right)  \sigma^{2}\left\vert A\left(  \rho\right)  \right\vert ^{2}\rho
d\rho,\\
\sigma^{2}\left\vert A\left(  \rho\right)  \right\vert ^{2}  &  =\frac{1}%
{2\pi}\int_{0}^{\infty}J_{0}\left(  \rho r\right)  \mathcal{C}_{2}\left(
r\right)  rdr,
\end{align*}
where $\mathcal{C}_{2}\left(  r\right)  =\operatorname*{Cov}\left(  X\left(
\underline{x}\right)  ,X\left(  \underline{y}\right)  \right)  $,
$r=\left\vert \underline{x}-\underline{y}\right\vert $.

The third order structure of a homogenous and isotropic stochastic field
$X\left(  \underline{x}\right)  $ is described by either the third order
cumulants (bicovariances) in spatial domain or the bispectrum in frequency
domain. The bicovariance of $X\left(  \underline{x}\right)  $ is%
\begin{equation}
\operatorname*{Cum}\left(  X\left(  \underline{x}_{1}\right)  ,X\left(
\underline{x}_{2}\right)  ,X\left(  \underline{x}_{3}\right)  \right)
=\operatorname*{Cum}\left(  X\left(  g\left(  \underline{x}_{1}-\underline
{x}_{3}\right)  \right)  X\left(  \left\vert \underline{x}_{2}-\underline
{x}_{3}\right\vert \underline{n}\right)  ,X\left(  0\right)  \right)
\label{Bicov}%
\end{equation}
where $g$ denotes the rotation carrying $\underline{x}_{2}-\underline{x}_{3}$
into the $\underline{n}=\left(  0,1\right)  $. The third order cumulant of the
stochastic spectral measure $Z\left(  d\underline{\omega}\right)  $ of the
homogenous field $X\left(  \underline{x}\right)  $ is given by
\begin{align*}
\operatorname*{Cum}\left(  Z\left(  d\underline{\omega}_{1}\right)  ,Z\left(
d\underline{\omega}_{2}\right)  ,Z\left(  d\underline{\omega}_{3}\right)
\right)   &  =\delta\left(  \Sigma_{1}^{3}\underline{\omega}_{k}\right)
S_{3}\left(  \underline{\omega}_{1},\underline{\omega}_{2},\underline{\omega
}_{3}\right)  d\underline{\omega}_{1}d\underline{\omega}_{2}d\underline
{\omega}_{3}\\
&  =\delta\left(  \Sigma_{1}^{3}\rho_{k}\underline{\widehat{\omega}}%
_{k}\right)  S_{3}\left(  \rho_{1},\rho_{2},\alpha_{3}\right)
{\textstyle\prod\limits_{k=1}^{3}}
\Omega\left(  d\underline{\widehat{\omega}}_{k}\right)  \rho_{k}d\rho_{k},
\end{align*}
where $\underline{\widehat{\omega}}_{k}=\underline{\omega}_{k}/\left\vert
\underline{\omega}_{k}\right\vert $. Now $S_{3}\left(  \rho_{1},\rho
_{2},\alpha_{3}\right)  $ depends on $0<\alpha_{3}<\pi$, in other words
depends on $\left(  \rho_{1},\rho_{2},\rho_{3}\right)  $, such that these
positive numbers form a triangle, see Figure \ref{Triangle}.

The bicovariance $\operatorname*{Cum}\left(  X\left(  \underline{x}%
_{1}\right)  ,X\left(  r_{2}\underline{n}\right)  ,X\left(  0\right)  \right)
$ depends on the lengths $r_{2}$, $r_{1}=\left\vert \underline{x}%
_{1}\right\vert $ and the angle $\varphi$ between them, this way a triangle is
defined with length of the third side $r_{3}$, such that $r_{3}^{2}=r_{1}%
^{2}+r_{2}^{2}-2r_{1}r_{2}\cos\left(  \varphi\right)  $. According to this
definition of $r_{1}$, we introduce $\mathcal{C}_{3}\left(  r_{1},r_{2}%
,r_{3}\right)  =\operatorname*{Cum}\left(  X\left(  0\right)  ,X\left(
r_{2}\underline{n}\right)  ,X\left(  \underline{x}_{3}\right)  \right)  $.
Similarly the bispectrum $S_{3}$ of the homogenous and isotropic stochastic
field $X\left(  \underline{x}\right)  $ depends on wave numbers $\left(
\rho_{1},\rho_{2},\rho_{3}\right)  $ such that $\rho_{1},\rho_{2},\rho_{3}$
should form a triangle. It has been shown, see \cite{Terdik2014Publi}, that
\[
\mathcal{C}_{3}\left(  \varphi,r_{2},r_{3}\right)  =2\iint\limits_{0}^{\infty
}\int_{0}^{\pi}\mathcal{T}_{3}\left(  \left.  \alpha,\rho_{2},\rho
_{3}\right\vert \varphi,r_{2},r_{3}\right)  S_{3}\left(  \rho_{1},\rho
_{2},\alpha\right)  d\alpha%
{\textstyle\prod\limits_{k=2}^{3}}
\rho_{k}d\rho_{k},
\]
where the function
\begin{equation}
\mathcal{T}_{3}\left(  \left.  \alpha,\rho_{2},\rho_{3}\right\vert
\varphi,r_{2},r_{3}\right)  =\sum_{\ell=-\infty}^{\infty}\cos\left(
\ell\varphi\right)  J_{\ell}\left(  \rho_{2}r_{2}\right)  J_{\ell}\left(
\rho_{3}r_{3}\right)  \cos\left(  \ell\alpha\right)  , \label{Transf_T}%
\end{equation}
gives the transformation of the bispectrum $S_{3}\left(  \rho_{1},\rho
_{2},\alpha\right)  $ into the bicovariance $\mathcal{C}_{3}\left(
\varphi,r_{2},r_{3}\right)  $. Notice that both angles $\varphi$ and $\alpha$
are related to the third sides $\rho_{1}$ and $r_{1}$ of the triangles,
defined by the wave numbers $\left(  \rho_{1},\rho_{2},\rho_{3}\right)  $ and
distances $\left(  r_{1},r_{2},r_{3}\right)  $. Distances $\left(  r_{1}%
,r_{2},r_{3}\right)  $ \textit{are not}\textbf{ }the norm of\textbf{ }$\left(
\underline{x}_{1},\underline{x}_{2},\underline{x}_{3}\right)  $ in
(\ref{Bicov}) but defined by $\left\vert \underline{x}_{1}-\underline{x}%
_{3}\right\vert $, $\left\vert \underline{x}_{2}-\underline{x}_{3}\right\vert
$ and the angle $\varphi$ is the one between the differences. By inversion of
the bicovariance function, we obtain the bispectrum
\begin{equation}
S_{3}\left(  \rho_{1},\rho_{2},\rho_{3}\right)  =\frac{1}{2\pi}\iint
\limits_{0}^{\infty}\int_{0}^{\pi}\mathcal{T}_{3}\left(  \left.  \alpha
,\rho_{2},\rho_{3}\right\vert \varphi,r_{2},r_{3}\right)  \mathcal{C}%
_{3}\left(  r_{1},r_{2},r_{3}\right)  d\varphi%
{\textstyle\prod\limits_{k=2}^{3}}
r_{k}dr_{k}. \label{Bisp_T_transf}%
\end{equation}

\section{Expression for Trispectrum of the Field}

The spectrum and bispectrum of a homogenous and isotropic stochastic field
have particular form and as it will be seen, considering trispectrum, they do
not show the general pattern for higher order spectra.

From now on we introduce a short notation for vectors using sets for indices,
for instance $\underline{\omega}_{1:4}$, where $1:4=\left(  1,2,3,4\right)  $,
denotes $\left(  \underline{\omega}_{1},\underline{\omega}_{2},\underline
{\omega}_{3},\underline{\omega}_{4}\right)  $, and so on.

Consider the spectral representation of the fourth order cumulant of a
homogenous field
\[
\operatorname*{Cum}\left(  X\left(  \underline{x}_{1}\right)  ,X\left(
\underline{x}_{2}\right)  ,X\left(  \underline{x}_{3}\right)  ,X\left(
\underline{x}_{4}\right)  \right)  =\underset{4}{\underbrace{\int
_{\mathbb{R}^{2}}\cdots\int_{\mathbb{R}^{2}}}}e^{i\left(  \Sigma_{1}%
^{4}\underline{x}_{k}\cdot\underline{\omega}_{k}\right)  }S_{4}\left(
\underline{\omega}_{1:4}\right)  \delta\left(  \Sigma_{1}^{4}\underline
{\omega}_{k}\right)
{\textstyle\prod\limits_{k=1}^{4}}
d\underline{\omega}_{k},
\]
and under isotropy assumption for each $g\in SO\left(  2\right)  $ we have in
addition
\begin{align*}
\operatorname*{Cum}\left(  X\left(  g\underline{x}_{1}\right)  ,X\left(
g\underline{x}_{2}\right)  ,X\left(  g\underline{x}_{3}\right)  ,X\left(
g\underline{x}_{4}\right)  \right)   &  =\underset{4}{\underbrace
{\int_{\mathbb{R}^{2}}\cdots\int_{\mathbb{R}^{2}}}}e^{i\left(  \Sigma_{1}%
^{4}\underline{x}_{k}\cdot\underline{\omega}_{k}\right)  }S_{4}\left(
g\underline{\omega}_{1:4}\right)  \delta\left(  \Sigma_{1}^{4}\underline
{\omega}_{k}\right)
{\textstyle\prod\limits_{k=1}^{4}}
d\underline{\omega}_{k}\\
&  =\operatorname*{Cum}\left(  X\left(  \underline{x}_{1}\right)  ,X\left(
\underline{x}_{2}\right)  ,X\left(  \underline{x}_{3}\right)  ,X\left(
\underline{x}_{4}\right)  \right)  ,
\end{align*}
hence $S_{4}\left(  \underline{\omega}_{1:4}\right)  =S_{4}\left(
\underline{\omega}_{1:3},-\Sigma_{1}^{3}\underline{\omega}_{k}\right)  $ and
at the same time $S_{4}\left(  g\underline{\omega}_{1:4}\right)  =S_{4}\left(
\underline{\omega}_{1:4}\right)  $. Now $\underline{\omega}_{1:4}$ is defined
by eight coordinates and if $\Sigma_{1}^{4}\underline{\omega}_{k}=0$, then
these vectors form a quadrilateral, in general. This quadrilateral has
invariants under rotations. In this way $S_{4}\left(  \underline{\omega}%
_{1:4}\right)  =S_{4}\left(  \alpha_{1},\rho_{2},\rho_{3},\rho_{4},\beta
_{2}\right)  $, see Figure \ref{Quadrilateral}.

Now, let us shift the vectors $\left(  \underline{x}_{1},\underline{x}%
_{2},\underline{x}_{3},\underline{x}_{4}\right)  $ by the vector
$-\underline{x}_{1}$, and rotate this new set of vectors $\left(
\underline{x}_{2}-\underline{x}_{1},\underline{x}_{3}-\underline{x}%
_{1},\underline{x}_{4}-\underline{x}_{1}\right)  =\left(  \underline{y}%
_{2},\underline{y}_{3},\underline{y}_{4}\right)  $ such that $\underline
{y}_{4}=\left\vert \underline{y}_{4}\right\vert \underline{n}$, where
$\underline{n}=\left(  1,0\right)  $, the locations $\left(  \underline
{0},\underline{y}_{2},\underline{y}_{3},\underline{y}_{4}\right)  $ will be
given by $r_{2}=\left\vert \underline{y}_{2}\right\vert $, $r_{3}=\left\vert
\underline{y}_{3}\right\vert $, $r_{4}$ $=\left\vert \underline{y}%
_{4}\right\vert $, together with angles $\left(  r_{2},\varphi_{2}\right)  $,
$\left(  r_{3},\varphi_{3}\right)  $. From now on we shall consider the
cumulants at locations defined by its invariants $\left(  \left(
r,\varphi\right)  _{2:3},r_{4}\right)  =\left(  \left(  r_{2},\varphi
_{2}\right)  ,\left(  r_{3},\varphi_{3}\right)  ,r_{4}\right)  $.

We use the invariance of the cumulants under the shift and rotation to obtain
\begin{multline*}
\operatorname*{Cum}\left(  X\left(  \underline{x}_{1}\right)  ,X\left(
\underline{x}_{2}\right)  ,X\left(  \underline{x}_{3}\right)  ,X\left(
\underline{x}_{4}\right)  \right)  =\operatorname*{Cum}\left(  X\left(
0\right)  ,X\left(  \underline{x}_{2}-\underline{x}_{1}\right)  ,X\left(
\underline{x}_{3}-\underline{x}_{1}\right)  ,X\left(  \underline{x}%
_{4}-\underline{x}_{1}\right)  ,\right) \\
=\operatorname*{Cum}\left(  X\left(  0\right)  ,X\left(  g\left(
\underline{x}_{2}-\underline{x}_{1}\right)  \right)  ,X\left(  g\left(
\underline{x}_{3}-\underline{x}_{1}\right)  \right)  ,X\left(  g\left(
\underline{x}_{4}-\underline{x}_{1}\right)  \right)  \right)
\end{multline*}
where $g$ denotes the rotation carrying the $\underline{x}_{4}-\underline
{x}_{1}$ into the $x$ axis. The general form of cumulants is
$\operatorname*{Cum}\left(  X\left(  0\right)  ,X\left(  \underline{x}%
_{2}\right)  ,X\left(  \underline{x}_{3}\right)  ,X\left(  r_{4}\underline
{n}\right)  \right)  $, where $\underline{x}_{2}$ and $\underline{x}_{3}$ are
arbitrary locations and $\underline{n}=\left(  1,0\right)  $. The fourth order
cumulants of a homogenous and isotropic stochastic field $X\left(
\underline{x}\right)  $ are determined by the quantities $r_{4}$,
$\underline{x}_{2}$ and $\underline{x}_{3}$, in other words by $\left(
\left(  r,\varphi\right)  _{2:3},r_{4}\right)  $, see Figure
\ref{Quadrilateral_x}. Let
\[
\mathcal{C}_{4}\left(  \left(  r_{2},\varphi_{2}\right)  ,\left(
r_{3},\varphi_{3}\right)  ,r_{4}\right)  =\operatorname*{Cum}\left(  X\left(
\underline{x}_{1}\right)  ,X\left(  \underline{x}_{2}\right)  ,X\left(
r_{3}\underline{n}\right)  ,X\left(  0\right)  \right)  ,
\]
and the trispectrum $S_{4}\left(  \alpha_{1},\rho_{2:4},\beta_{2}\right)  $ is
given on the domain of variables $\left(  \alpha_{1},\rho_{2},\rho_{3}%
,\rho_{4},\beta_{2}\right)  $, where $\alpha_{1},\beta_{2}\in\left(
0,\pi\right)  $, and $0<\rho_{2},\rho_{3},\rho_{4}$.

Let
\begin{align*}
&  \mathcal{T}_{4}\left(  \left.  \alpha_{1},\rho_{2:4},\beta_{2}\right\vert
\left(  r,\varphi\right)  _{2:3},r_{4}\right) \\
&  =\sum_{\ell_{2},\ell_{3}=-\infty}^{\infty}e^{i\left(  \ell_{2}\varphi
_{2}+\ell_{3}\varphi_{3}\right)  }J_{\ell_{2}}\left(  \rho_{2}r_{2}\right)
J_{\ell_{3}}\left(  \rho_{3}r_{3}\right)  J_{\ell_{2}+\ell_{3}}\left(
\rho_{4}r_{4}\right)  \cos\left(  \ell_{2}\alpha_{1}\right)  \cos\left(
\ell_{2}\alpha_{3}-\ell_{3}\beta_{2}\right)  ,
\end{align*}
where the angle $\alpha_{3}$ is determined by $\rho_{2},\rho_{3},\rho_{4}$,
i.e. $\left(  \rho_{2}^{2}+\rho_{4}^{2}-\rho_{3}^{2}\right)  /\left(
2\rho_{2}\rho_{4}\right)  =\cos\alpha_{3}$, see Figure \ref{Quadrilateral}.

\begin{theorem}
\label{Theo_trisp}Let $X\left(  \underline{x}\right)  $ be a homogenous and
isotropic stochastic field on the plane then
\begin{align*}
\mathcal{C}_{4}\left(  \left(  r,\varphi\right)  _{2:3},r_{4}\right)   &
=4\iiint\limits_{0}^{\infty}\iint_{0}^{\pi}\mathcal{T}_{4}\left(  \left.
\alpha_{1},\rho_{2:4},\beta_{2}\right\vert \left(  r,\varphi\right)
_{2:3},r_{4}\right) \\
&  \times S_{4}\left(  \alpha_{1},\rho_{2:4},\beta_{2}\right)
{\textstyle\prod\limits_{k=2}^{4}}
\rho_{k}d\rho_{k}d\alpha_{1}d\beta_{2},
\end{align*}
In return
\begin{align}
S_{4}\left(  \alpha_{1},\rho_{2:4},\beta_{2}\right)   &  =\frac{1}{\left(
2\pi\right)  ^{4}}\iiint\limits_{0}^{\infty}\iint_{0}^{2\pi}\mathcal{T}%
_{4}\left(  \left.  \alpha_{1},\rho_{2:4},\beta_{2}\right\vert \left(
r,\varphi\right)  _{2:3},r_{4}\right) \nonumber\\
&  \times\mathcal{C}_{4}\left(  \left(  r,\varphi\right)  _{2:3},r_{4}\right)
%
{\textstyle\prod\limits_{k=2}^{4}}
r_{k}dr_{k}d\varphi_{2}d\varphi_{3}, \label{Trsp_tricov}%
\end{align}
unless the integrals exist.
\end{theorem}

\begin{proof}
We apply the series representation (\ref{SeriesRepr_X}) of $X\left(
\underline{x}\right)  $, and rewrite it for particular cases $\underline
{n}=\left(  -1,0\right)  $
\begin{align}
X\left(  r\underline{n}\right)   &  =\sum_{\ell=-\infty}^{\infty}\int
_{0}^{\infty}J_{\ell}\left(  \rho r\right)  Z_{\ell}\left(  \rho d\rho\right)
,\label{X_north}\\
X\left(  \underline{0}\right)   &  =\int_{\mathbb{R}^{2}}Z\left(
d\underline{\omega}\right)  \label{X_0}\\
&  =\int_{0}^{\infty}Z_{0}\left(  \rho d\rho\right)  .\nonumber
\end{align}
We obtain
\begin{multline*}
\operatorname*{Cum}\left(  X\left(  0\right)  ,X\left(  \underline{x}%
_{2}\right)  ,X\left(  \underline{x}_{3}\right)  ,X\left(  r_{4}\underline
{n}\right)  \right)  =\sum_{\ell_{2},\ell_{3},\ell_{4}=-\infty}^{\infty
}e^{i\left(  \ell_{2}\varphi_{2}+\ell_{3}\varphi_{3}\right)  }\\
\times\iiiint\limits_{0}^{\infty}J_{\ell_{2}}\left(  \rho_{2}r_{2}\right)
J_{\ell_{3}}\left(  \rho_{3}r_{3}\right)  J_{\ell_{4}}\left(  \rho_{4}%
r_{4}\right)  \operatorname*{Cum}\left(  Z_{0}\left(  \rho_{1}d\rho
_{1}\right)  ,Z_{\ell_{2}}\left(  \rho_{2}d\rho_{2}\right)  ,Z_{\ell_{3}%
}\left(  \rho_{3}d\rho_{3}\right)  ,Z_{\ell_{4}}\left(  \rho_{4}d\rho
_{4}\right)  \right)  \\
=\sum_{\ell_{2},\ell_{3}=-\infty}^{\infty}e^{i\left(  \ell_{2}\varphi_{2}%
+\ell_{3}\varphi_{3}\right)  }\iiiint\limits_{0}^{\infty}J_{\ell_{2}}\left(
\rho_{2}r_{2}\right)  J_{\ell_{3}}\left(  \rho_{3}r_{3}\right)  J_{-\left(
\ell_{2}+\ell_{3}\right)  }\left(  \rho_{4}r_{4}\right)  \\
\times\operatorname*{Cum}\left(  Z_{0}\left(  \rho_{1}d\rho_{1}\right)
,Z_{\ell_{2}}\left(  \rho_{2}d\rho_{2}\right)  ,Z_{\ell_{3}}\left(  \rho
_{3}d\rho_{3}\right)  ,Z_{-\left(  \ell_{2}+\ell_{3}\right)  }\left(  \rho
_{4}d\rho_{4}\right)  \right)  \\
=\sum_{\ell_{2},\ell_{3}=-\infty}^{\infty}e^{i\left(  \ell_{2}\varphi_{2}%
+\ell_{3}\varphi_{3}\right)  }\iiiint\limits_{0}^{\infty}J_{\ell_{2}}\left(
\rho_{2}r_{2}\right)  J_{\ell_{3}}\left(  \rho_{3}r_{3}\right)  J_{\ell
_{2}+\ell_{3}}\left(  \rho_{4}r_{4}\right)  \\
\times\left(  -1\right)  ^{\ell_{2}+\ell_{3}}\operatorname*{Cum}\left(
Z_{0}\left(  \rho_{1}d\rho_{1}\right)  ,Z_{\ell_{2}}\left(  \rho_{2}d\rho
_{2}\right)  ,Z_{\ell_{3}}\left(  \rho_{3}d\rho_{3}\right)  ,Z_{-\left(
\ell_{2}+\ell_{3}\right)  }\left(  \rho_{4}d\rho_{4}\right)  \right)  ,
\end{multline*}
in polar coordinates. The fourth order cumulant of the stochastic spectral
measure $Z\left(  d\underline{\omega}\right)  $ according to a homogenous
field $X\left(  \underline{x}\right)  $ fulfils the following equation
\[
\operatorname*{Cum}\left(  Z\left(  d\underline{\omega}_{1}\right)  ,Z\left(
d\underline{\omega}_{2}\right)  ,Z\left(  d\underline{\omega}_{3}\right)
,Z\left(  d\underline{\omega}_{4}\right)  \right)  =\delta\left(  \Sigma
_{1}^{4}\underline{\omega}_{k}\right)  S_{4}\left(  \underline{\omega}%
_{1:4}\right)
{\textstyle\prod\limits_{k=1}^{4}}
d\underline{\omega}_{k},
\]
and the stochastic spectral measures $Z_{\ell}\left(  \rho d\rho\right)  $ are
connected to $Z\left(  d\underline{\omega}\right)  $ by (\ref{measure_Z_l}) in
frequency domain, hence
\begin{align}
&  \operatorname*{Cum}\left(  Z_{0}\left(  \rho_{1}d\rho_{1}\right)
,Z_{\ell_{2}}\left(  \rho_{2}d\rho_{2}\right)  ,Z_{\ell_{3}}\left(  \rho
_{3}d\rho_{3}\right)  ,Z_{-\left(  \ell_{2}+\ell_{3}\right)  }\left(  \rho
_{4}d\rho_{4}\right)  \right)  \nonumber\\
\hspace{0.5in} &  =4\left(  -1\right)  ^{\ell_{2}+\ell_{3}}\int_{0}^{\pi}%
\frac{\delta\left(  \bigtriangleup|\rho_{1},\rho_{2},\kappa\right)  }{\rho
_{1}\kappa\sin\alpha_{2}}\cos\left(  \ell_{2}\alpha_{1}\right)  \cos\left(
\ell_{2}\alpha_{3}+\ell_{3}\beta_{2}\right)  S_{4}\left(  \alpha_{1}%
,\rho_{2:4},\beta_{2}\right)  d\beta_{2}%
{\textstyle\prod\limits_{k=1}^{4}}
\rho_{k}d\rho_{k}\label{Cum4_Z_L1}%
\end{align}
where $\underline{\widehat{\omega}}_{k}=\underline{\omega}_{k}/\left\vert
\underline{\omega}_{k}\right\vert =\left(  \cos\eta_{k},\sin\eta_{k}\right)  $
defines the angle $\eta_{k}$. Notice that the cumulants \newline%
$\operatorname*{Cum}\left(  X\left(  0\right)  ,X\left(  \underline{x}%
_{2}\right)  ,X\left(  \underline{x}_{3}\right)  ,X\left(  r_{4}\underline
{n}\right)  \right)  $ are given in terms of three distances and two angles
$\left(  \left(  r,\varphi\right)  _{2:3},r_{4}\right)  $, see Figure
\ref{Quadrilateral_x}.
\begin{figure}
[ptbh]
\begin{center}
\includegraphics[
natheight=4.019700in,
natwidth=7.400200in,
height=3.2603in,
width=5.9819in
]%
{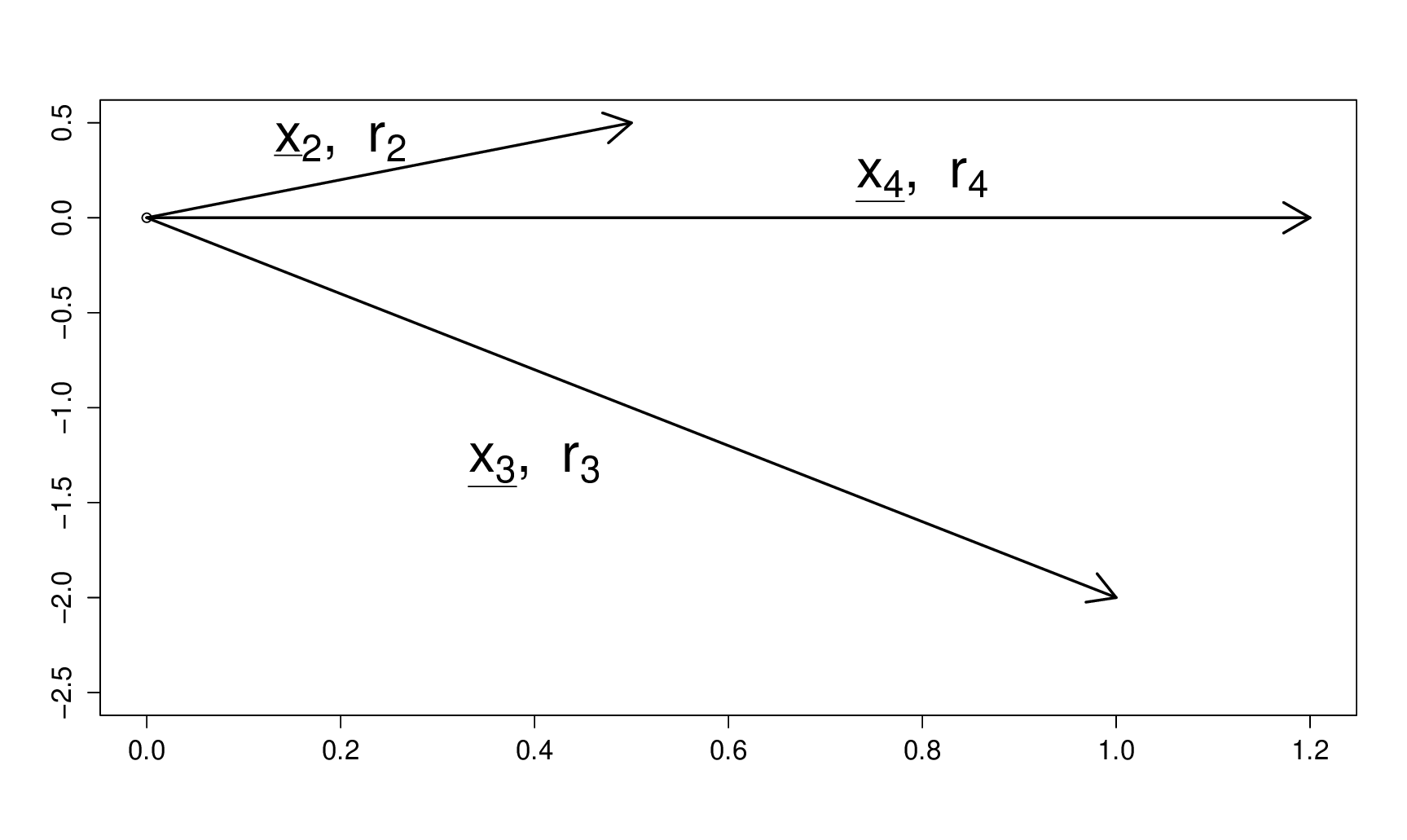}%
\caption{Locations on the plane}%
\label{Quadrilateral_x}%
\end{center}
\end{figure}
The function $\mathcal{C}_{4}\left(  \left(  \left(  r,\varphi\right)
_{2:3},r_{4}\right)  \right)  $ is expressed as
\begin{multline*}
\operatorname*{Cum}\left(  X\left(  0\right)  ,X\left(  \underline{x}%
_{2}\right)  ,X\left(  \underline{x}_{3}\right)  ,X\left(  r_{4}\underline
{n}\right)  \right)  \\
=4\sum_{\ell_{2},\ell_{3}=-\infty}^{\infty}e^{i\left(  \ell_{2}\varphi
_{2}+\ell_{3}\varphi_{3}\right)  }\iiiint\limits_{0}^{\infty}J_{\ell_{2}%
}\left(  \rho_{2}r_{2}\right)  J_{\ell_{3}}\left(  \rho_{3}r_{3}\right)
J_{\ell_{2}+\ell_{3}}\left(  \rho_{4}r_{4}\right)  \\
\times\frac{\delta\left(  \bigtriangleup|\rho_{1},\rho_{2},\kappa\right)
}{\rho_{2}\kappa\sin\alpha_{1}}e^{-i\ell_{2}\left(  \alpha_{1}-\alpha
_{3}\right)  -i\ell_{3}\beta_{2}\ }S_{4}\left(  \alpha_{1},\rho_{2:4}%
,\beta_{2}\right)  d\beta_{2}%
{\textstyle\prod\limits_{k=1}^{4}}
\rho_{k}d\rho_{k}\\
=4\sum_{\ell_{2},\ell_{3}=-\infty}^{\infty}e^{i\left(  \ell_{2}\varphi
_{2}+\ell_{3}\varphi_{3}\right)  }\iiint\limits_{0}^{\infty}J_{\ell_{2}%
}\left(  \rho_{2}r_{2}\right)  J_{\ell_{3}}\left(  \rho_{3}r_{3}\right)
J_{\ell_{2}+\ell_{3}}\left(  \rho_{4}r_{4}\right)  \\
\iint\limits_{0}^{\pi}\cos\left(  \ell_{2}\alpha_{1}\right)  \cos\left(
\ell_{2}\alpha_{3}+\ell_{3}\beta_{2}\right)  S_{4}\left(  \alpha_{1}%
,\rho_{2:4},\beta_{2}\right)  d\alpha_{1}d\beta_{2}%
{\textstyle\prod\limits_{k=2}^{4}}
\rho_{k}d\rho_{k},
\end{multline*}
$\rho_{1}d\rho_{1}=\kappa\rho_{2}\sin\left(  \alpha_{1}\right)  d\alpha_{1}$
where $\kappa=\sqrt{\rho_{3}^{2}+\rho_{4}^{2}-2\rho_{3}\rho_{4}\cos\beta_{2}}%
$, and $\alpha_{2}=\arccos\left[  \left(  \rho_{3}^{2}-\rho_{4}^{2}-\kappa
^{2}\right)  /2\kappa\rho_{4}\right]  $, hence $\alpha_{3}$ is determined by
$\rho_{3}$, $\rho_{4}$ and $\beta_{2}$. The result of the above summation is
real therefore the imaginary part is zero. \newline Now, for proving
(\ref{Trsp_tricov}) consider the integral
\begin{align*}
&  \iint\limits_{0}^{2\pi}\iiint\limits_{0}^{\infty}\mathcal{T}_{4}\left(
\left.  \alpha_{1},\rho_{2:4},\beta_{2}\right\vert \left(  r,\varphi\right)
_{2:3},r_{4}\right)  \mathcal{C}_{4}\left(  \left(  r,\varphi\right)
_{2:3},r_{4}\right)  d\varphi_{2}d\varphi_{3}%
{\textstyle\prod\limits_{k=2}^{4}}
r_{k}dr_{k}\\
&  =4\iint\limits_{0}^{2\pi}\iiint\limits_{0}^{\infty}\mathcal{T}_{4}\left(
\left.  \alpha_{1},\rho_{2:4},\beta_{2}\right\vert \left(  r,\varphi\right)
_{2:3},r_{4}\right)  \\
&  \iint\limits_{0}^{2\pi}\iiint\limits_{0}^{\infty}\mathcal{T}_{4}\left(
\left.  \alpha_{1}^{\prime},\rho_{2:4}^{\prime},\beta_{2}^{\prime}\right\vert
\left(  r,\varphi\right)  _{2:3},r_{4}\right)  S_{4}\left(  \alpha_{1}%
^{\prime},\rho_{2:4}^{\prime},\beta_{2}^{\prime}\right)  d\alpha_{1}^{\prime
}d\beta_{2}^{\prime}%
{\textstyle\prod\limits_{k=2}^{4}}
\rho_{k}^{\prime}d\rho_{k}^{\prime}d\varphi_{2}d\varphi_{3}%
{\textstyle\prod\limits_{k=2}^{4}}
r_{k}dr_{k}\\
&  =4\left(  2\pi\right)  ^{2}\sum_{\ell_{2},\ell_{3}=-\infty}^{\infty}%
\iiint\limits_{0}^{\infty}\iiint\limits_{0}^{\infty}J_{\ell_{2}}\left(
\rho_{2}^{\prime}r_{2}\right)  J_{\ell_{3}}\left(  \rho_{3}^{\prime}%
r_{3}\right)  J_{\ell_{2}+\ell_{3}}\left(  \rho_{4}^{\prime}r_{4}\right)
J_{\ell_{2}}\left(  \rho_{2}r_{2}\right)  J_{\ell_{3}}\left(  \rho_{3}%
r_{3}\right)  J_{\ell_{2}+\ell_{3}}\left(  \rho_{4}r_{4}\right)
{\textstyle\prod\limits_{k=2}^{4}}
r_{k}dr_{k}\\
&  \times\iint\limits_{0}^{\pi}\cos\left(  \ell_{2}\alpha_{1}^{\prime}\right)
\cos\left(  \ell_{2}\alpha_{3}^{\prime}+\ell_{3}\beta_{2}^{\prime}\right)
\cos\left(  \ell_{2}\alpha_{1}\right)  \cos\left(  \ell_{2}\alpha_{3}+\ell
_{3}\beta_{2}\right)  S_{4}\left(  \alpha_{1}^{\prime},\rho_{2:4}^{\prime
},\beta_{2}^{\prime}\right)  d\alpha_{1}^{\prime}d\beta_{2}^{\prime}%
{\textstyle\prod\limits_{k=2}^{4}}
\rho_{k}^{\prime}d\rho_{k}^{\prime}\\
&  =4\left(  2\pi\right)  ^{2}\iint\limits_{0}^{\pi}\sum_{\ell_{2},\ell
_{3}=-\infty}^{\infty}\cos\left(  \ell_{2}\alpha_{1}^{\prime}\right)
\cos\left(  \ell_{2}\alpha_{3}^{\prime}+\ell_{3}\beta_{2}^{\prime}\right)
\cos\left(  \ell_{2}\alpha_{1}\right)  \cos\left(  \ell_{2}\alpha_{3}+\ell
_{3}\beta_{2}\right)  S_{4}\left(  \alpha_{1}^{\prime},\rho_{2:4},\beta
_{2}^{\prime}\right)  d\alpha_{1}^{\prime}d\beta_{2}^{\prime}\\
&  =\left(  2\pi\right)  ^{4}S_{4}\left(  \alpha_{1},\rho_{2:4},\beta
_{2}\right)  ,
\end{align*}

To show the last equality, one can turn cosine to exponential and get the
result, since both $\beta_{2}$ and $\beta_{2}^{\prime}$ are positive,
similarly $\alpha_{3}\ $and $\alpha_{3}^{\prime}$. If $\beta_{2}=\beta
_{2}^{\prime}$, then $\alpha_{3}=\alpha_{3}^{\prime}$ follows, finally
$\alpha_{1}=\alpha_{1}^{\prime}$.
\end{proof}

\section{ Expression for Higher Order Spectra}

The Theorem \ref{Theo_trisp} can be generalized for higher order spectra. Put
$\rho_{2:p}=\left(  \rho_{2},\ldots,\rho_{p}\right)  $, $\beta_{1:p-3,2}%
=\left(  \beta_{1,2},\ldots,\beta_{p-3,2}\right)  $, $\left(  r,\varphi
\right)  _{2:p-1},=\left(  \left(  r_{2},\varphi_{2}\right)  ,\ldots,\left(
r_{p-1},\varphi_{p-1}\right)  \right)  $, and define the transformation
\begin{multline*}
\mathcal{T}_{p}\left(  \left.  \alpha_{1},\rho_{2:p},\beta_{1:p-3,2}%
\right\vert \left(  r,\varphi\right)  _{2:p-1},r_{p}\right) \\
=\sum_{\ell_{2},\ldots,\ell_{p-1}=-\infty}^{\infty}J_{\Sigma_{1}^{p-1}\ell
_{k}}\left(  \rho_{p}r_{p}\right)  \prod\limits_{k=2}^{p-1}e^{i\ell_{k}%
\varphi_{k}}J_{\ell_{k}}\left(  \rho_{k}r_{k}\right)  \cos\left(  \alpha_{k-1}%
{\textstyle\sum_{j=2}^{k-1}}
\ell_{j}-\ell_{k}\beta_{k+1}\right)  ,
\end{multline*}
where $\alpha_{1}%
{\textstyle\sum_{j=2}^{1}}
\ell_{j}=0$.

\begin{theorem}
Let $X\left(  \underline{x}\right)  $ be a homogenous and isotropic stochastic
field on the plane then%
\begin{align*}
&  \mathcal{C}_{p}\left(  \left(  r,\varphi\right)  _{2:p-1},r_{p}\right) \\
&  =2^{p-2}\int\nolimits_{0}^{\infty}\cdots\int\nolimits_{0}^{\infty}\int
_{0}^{\pi}\cdots\int_{0}^{\pi}\mathcal{T}_{p}\left(  \left.  \alpha_{1}%
,\rho_{2:p},\beta_{1:p-3,2}\right\vert \left(  r,\varphi\right)
_{2:p-1},r_{p}\right) \\
&  \times S_{p}\left(  \alpha_{1},\rho_{2:p},\beta_{1:p-3,2}\right)
{\textstyle\prod\limits_{k=2}^{p}}
\rho_{k}d\rho_{k}d\alpha_{1}\prod\limits_{k=1}^{p-3}d\beta_{k,2}.
\end{align*}
In return
\begin{align}
S_{p}\left(  \alpha_{1},\rho_{2:p},\beta_{1:p-3,2}\right)   &  =\frac
{1}{\left(  2\pi\right)  ^{p-2}}\int\nolimits_{0}^{\infty}\cdots
\int\nolimits_{0}^{\infty}\int_{0}^{\pi}\cdots\int_{0}^{\pi}\mathcal{T}%
_{p}\left(  \left.  \alpha_{1},\rho_{2:p},\beta_{1:p-3,2}\right\vert \left(
r,\varphi\right)  _{2:p-1},r_{p}\right) \nonumber\\
&  \times\mathcal{C}_{p}\left(  \left(  r,\varphi\right)  _{2:p-1}%
,r_{p}\right)  r_{p}dr_{p}%
{\textstyle\prod\limits_{k=2}^{p-1}}
r_{k}dr_{k}d\varphi_{k},
\end{align}
unless the integrals exist.
\end{theorem}

\begin{proof}
The argument of obtaining higher order spectra is similar to the evaluation of
the trispectrum, the only difference is that instead of using Lemma
\ref{Lemma_Jprod4} one has to use Theorem \ref{Theorem_Jprod}.
\end{proof}

\appendix

\section{Some Integrals of Bessel functions}

One of the key formula necessary for deriving an expression for the bispectrum
(see \cite{Terdik2014Publi}), is the integral
\begin{align*}
\int_{0}^{\infty}J_{0}\left(  \rho_{1}\lambda\right)  J_{\ell}\left(  \rho
_{2}\lambda\right)  J_{\ell}\left(  \rho_{3}\lambda\right)  \lambda d\lambda
&  =\frac{\cos\left(  \ell\arccos\left(  R\right)  \right)  }{\pi\rho_{2}%
\rho_{3}\sqrt{1-R^{2}}}\\
&  =\frac{\cos\left(  \ell\alpha_{1}\right)  }{\pi\rho_{2}\rho_{3}\sin
\alpha_{1}},
\end{align*}
where $\rho_{1}^{2}=\rho_{2}^{2}+\rho_{3}^{2}-2\rho_{2}\rho_{3}\cos\alpha_{1}$
and $R=\left(  \rho_{2}^{2}+\rho_{3}^{2}-\rho_{1}^{2}\right)  /\left(
2\rho_{2}\rho_{3}\right)  =\cos\alpha_{1}$ (see \cite{Prudnikov14} Tom. II,
2.12.41.16). This expression is a special case of the following result.%
\begin{figure}
[ptbh]
\begin{center}
\includegraphics[
height=3.3658in,
width=6.1765in
]%
{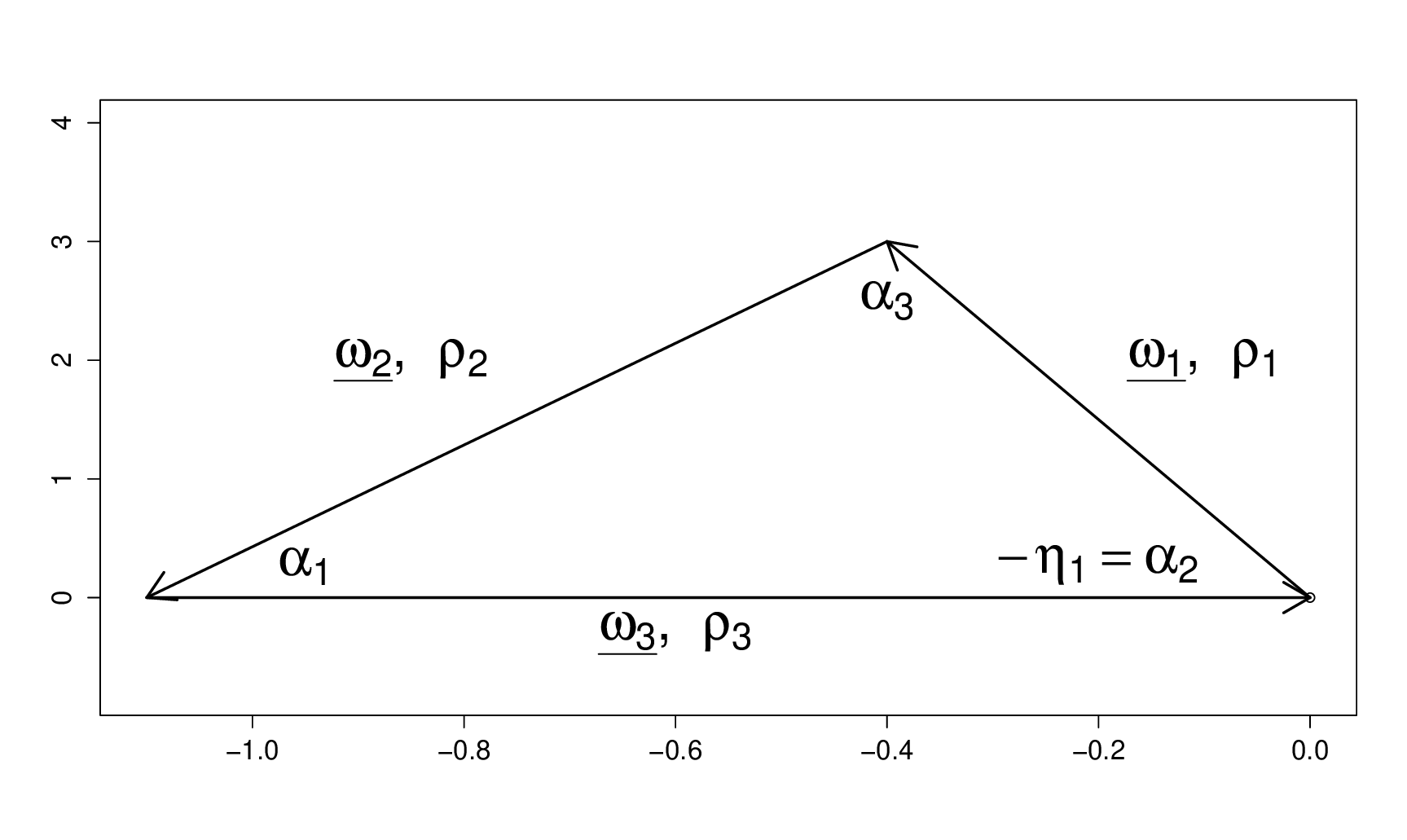}%
\caption{Triangle}%
\label{Triangle}%
\end{center}
\end{figure}

\begin{lemma}
\label{Lemma_Jprod3} Let $\rho_{k}>0$, $\left\vert \rho_{2}-\rho
_{3}\right\vert \leq\rho_{1}\leq\rho_{2}+\rho_{3}$, and $\rho_{1}^{2}=\rho
_{2}^{2}+\rho_{3}^{2}-2\rho_{2}\rho_{3}\cos\alpha_{1}$, then
\[
\int_{0}^{\infty}J_{\ell_{1}}\left(  \rho_{1}\lambda\right)  J_{\ell_{2}%
}\left(  \rho_{2}\lambda\right)  J_{\ell_{1}+\ell_{2}}\left(  \rho_{3}%
\lambda\right)  \lambda d\lambda=\frac{\cos\left(  \ell_{1}\alpha_{2}-\ell
_{2}\alpha_{1}\right)  }{\pi\rho_{2}\rho_{3}\sin\alpha_{1}},
\]
otherwise if $\rho_{1}$, $\rho_{2}\ $and $\rho_{3}$ do not form a triangle
then the integral is zero.
\end{lemma}

\begin{proof}
Formulae for integrals of Bessel functions require care and attention, see for
instance \cite{Vilenkin1968} p224 and the Addition Theorem \cite{Korenev2002},
p27. The assumptions imply that a triangle according to $\left(  \rho_{1}%
,\rho_{2},\rho_{3}\right)  $ can be formed, see Figure \ref{Triangle}.
Formally the following relations are valid; $\rho_{1}^{2}=\rho_{2}^{2}%
+\rho_{3}^{2}-2\rho_{2}\rho_{3}\cos\alpha_{1}$, $\rho_{2}-\rho_{3}\cos
\alpha_{1}=\rho_{1}\cos\alpha_{3}$, $\rho_{3}\sin\alpha_{1}=\rho_{1}\sin
\alpha_{3}$, \cite{Erdelyi11b}, T2, p54 , equivalently, $\sqrt{4\rho_{2}%
^{2}\rho_{3}^{2}-\left(  \rho^{2}-\rho_{2}^{2}-\rho_{3}^{2}\right)  ^{2}%
}=2\rho_{2}\rho_{3}\sin\alpha_{1}$, $\alpha_{1}\in\left(  0,\pi\right)  $. Let
us start with the Graf's Addition Theorem%
\[
e^{i\ell_{1}\alpha_{2}}J_{\ell_{1}}\left(  \rho_{1}\right)  =\sum_{m=-\infty
}^{\infty}J_{m}\left(  \rho_{2}\right)  J_{m+\ell_{1}}\left(  \rho_{3}\right)
e^{im\alpha_{1}}.
\]
The system $e^{im\alpha}$ is orthogonal on the $\left[  0,2\pi\right]  $, but
the angle $\alpha_{1}$ is changing on interval $\left[  0,\pi\right]  $, in
this case we have the integral
\[
\int_{0}^{\pi}e^{i\left(  m-\ell_{2}\right)  \alpha}d\alpha=\left\{
\begin{array}
[c]{ccc}%
\pi & \text{if } & m=\ell_{2},\\
\frac{i}{m-\ell_{2}}\left(  1-\left(  -1\right)  ^{m-\ell_{2}}\right)  &
\text{if } & m\neq\ell_{2},
\end{array}
\right.
\]
hence
\begin{multline*}
\int_{0}^{\pi}e^{i\ell_{1}\alpha_{2}}J_{\ell_{1}}\left(  \rho_{1}\right)
e^{-i\ell_{2}\alpha_{1}}d\alpha_{1}=\int_{0}^{\pi}\sum_{m=-\infty}^{\infty
}J_{m}\left(  \rho_{2}\right)  J_{m+\ell_{1}}\left(  \rho_{3}\right)
e^{i\left(  m-\ell_{2}\right)  \alpha_{1}}d\alpha_{1}\\
=\pi J_{\ell_{2}}\left(  \rho_{2}\right)  J_{\ell_{1}+\ell_{2}}\left(
\rho_{3}\right)  +2i\sum_{k=-\infty}^{\infty}\frac{1}{2k+1}J_{2k+1+\ell_{2}%
}\left(  \rho_{2}\right)  J_{2k+1+\ell_{1}+\ell_{2}}\left(  \rho_{3}\right)  .
\end{multline*}
The real part of the above equality provides
\begin{equation}
\int_{0}^{\pi}\cos\left(  \ell_{1}\alpha_{2}-\ell_{2}\alpha_{1}\right)
J_{\ell_{1}}\left(  \lambda\rho_{1}\right)  d\alpha_{1}=\pi J_{\ell_{2}%
}\left(  \lambda\rho_{2}\right)  J_{\ell_{1}+\ell_{2}}\left(  \lambda\rho
_{3}\right)  . \label{equ_Bess2}%
\end{equation}
Now, integrate over $\lambda d\lambda$ and applying the formula
(\ref{Dirac_Bessel}) we get
\begin{align*}
&  \hspace{-1in}\int_{0}^{\infty}J_{\ell_{1}}\left(  \rho_{1}\lambda\right)
J_{\ell_{2}}\left(  \rho_{2}\lambda\right)  J_{\ell_{1}+\ell_{2}}\left(
\rho_{3}\lambda\right)  \lambda d\lambda\\
&  =\int_{0}^{\infty}J_{\ell_{1}}\left(  \rho_{1}\lambda\right)  \frac{1}{\pi
}\int_{0}^{\pi}\cos\left(  \ell_{1}\alpha_{2}-\ell_{2}\gamma\right)
J_{\ell_{1}}\left(  \rho\lambda\right)  d\gamma\lambda d\lambda\\
&  =\frac{1}{\pi}\int_{\left\vert \rho_{2}-\rho_{3}\right\vert }^{\rho
_{2}+\rho_{3}}\int_{0}^{\infty}J_{\ell_{1}}\left(  \rho_{1}\lambda\right)
J_{\ell_{1}}\left(  \rho\lambda\right)  \lambda d\lambda\frac{\cos\left(
\ell_{1}\alpha_{2}-\ell_{2}\gamma\right)  \rho d\rho}{\rho_{2}\rho_{3}%
\sin\gamma}\\
&  =\frac{1}{\pi}\int_{\left\vert \rho_{2}-\rho_{3}\right\vert }^{\rho
_{2}+\rho_{3}}\frac{\cos\left(  \ell_{1}\alpha_{2}-\ell_{2}\gamma\right)
}{\rho_{2}\rho_{3}\sin\gamma}\frac{\delta\left(  \rho_{1}-\rho\right)  }%
{\rho_{1}}\rho d\rho\\
&  =\frac{\cos\left(  \ell_{1}\alpha_{2}-\ell_{2}\alpha_{1}\right)  }{\pi
\rho_{2}\rho_{3}\sin\alpha_{1}}.
\end{align*}
The integral is zero if the inequality $\left\vert \rho_{2}-\rho
_{3}\right\vert \leq\rho_{1}\leq\rho_{2}+\rho_{3}$ is not satisfied,
\cite{Vilenkin1968} p. 224.
\end{proof}

We consider a quadrilateral according to the wave numbers $\left(  \alpha
_{1},\rho_{2},\rho_{3},\rho_{4}\right)  $ defined by two triangles $\left(
\rho_{1},\rho_{2},\kappa\right)  $ and $\left(  \kappa,\rho_{3},\rho
_{4}\right)  $ where $\kappa=\left\vert \underline{\kappa}\right\vert $ is the
diagonal and $\rho_{j}=\left\vert \underline{\omega}_{j}\right\vert $, see
Figure \ref{Quadrilateral}. In other words $\left(  \underline{\omega}%
_{1},\underline{\omega}_{2},\underline{\kappa}\right)  $ and $\left(
\underline{\omega}_{3},\underline{\omega}_{4},-\underline{\kappa}\right)  $
are triangulars and their sides $\left(  \rho_{1},\rho_{2},\kappa\right)  $
and $\left(  \kappa,\rho_{3},\rho_{4}\right)  $ fulfil the triangle relation,
i.e. the assumption
\[
\max\left(  \left\vert \rho_{2}-\rho_{1}\right\vert ,\left\vert \rho_{4}%
-\rho_{3}\right\vert \right)  <\kappa<\min\left(  \rho_{1}+\rho_{2},\rho
_{3}+\rho_{4}\right)  ,
\]
fulfils, see Figure \ref{Quadrilateral}.
\begin{figure}
[ptbh]
\begin{center}
\includegraphics[
height=3.6936in,
width=6.7724in
]%
{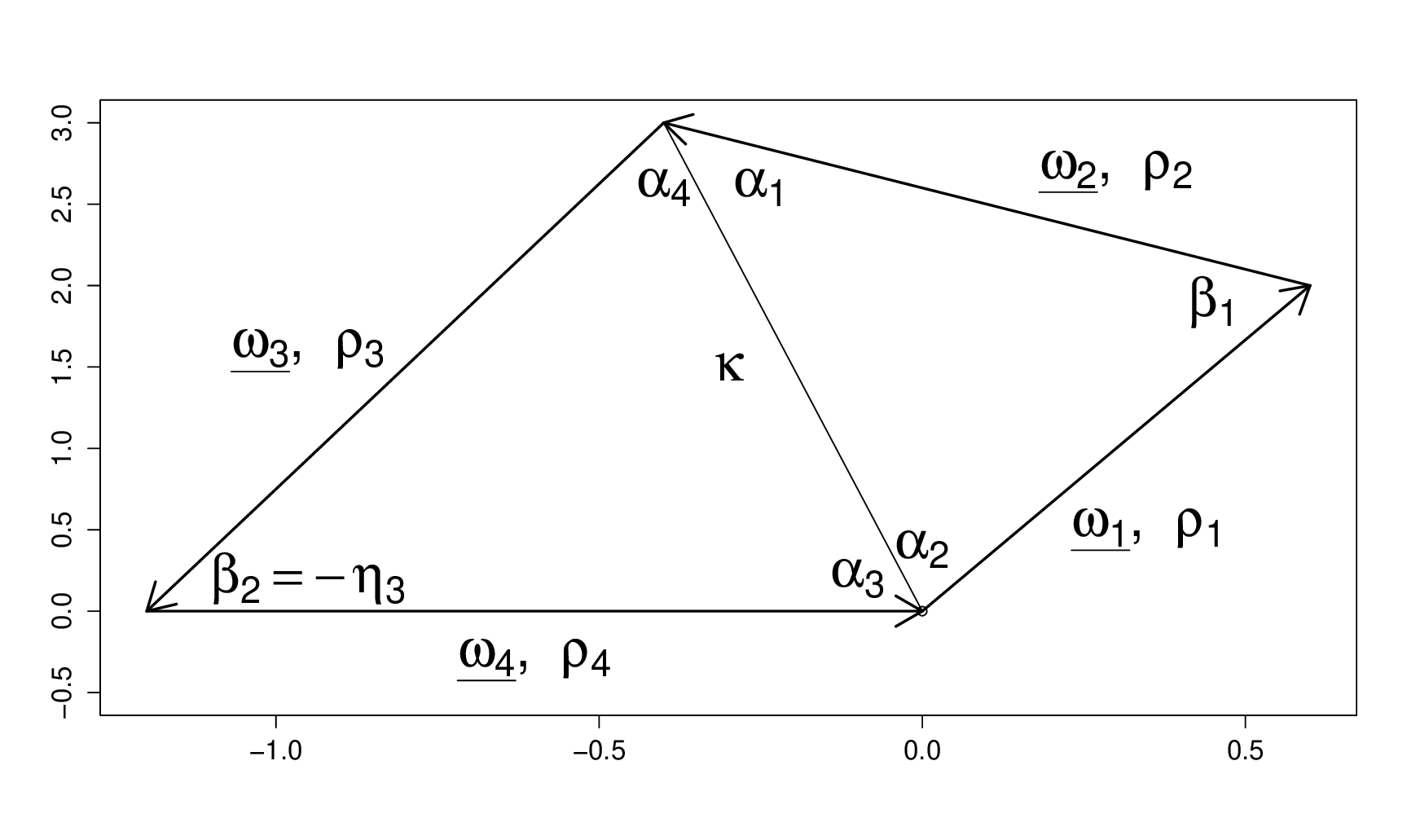}%
\caption{Quadrilateral}%
\label{Quadrilateral}%
\end{center}
\end{figure}

\begin{lemma}
\label{Lemma_Jprod4} Assume $\kappa^{2}=\rho_{3}^{2}+\rho_{4}^{2}-2\rho
_{3}\rho_{4}\cos\beta_{2}$, $\beta_{2}\in\left(  0,\pi\right)  $ and $\left(
\rho_{1},\rho_{2},\kappa\right)  $ defines a triangle, see Figure
\ref{Quadrilateral}, then
\begin{multline*}
\int_{0}^{\infty}J_{\ell_{1}}\left(  \rho_{1}\lambda\right)  J_{\ell_{2}%
}\left(  \rho_{2}\lambda\right)  J_{\ell_{3}}\left(  \rho_{3}\lambda\right)
J_{\ell_{1}+\ell_{2}+\ell_{3}}\left(  \rho_{4}\lambda\right)  \lambda
d\lambda\\
=\frac{1}{\pi^{2}}\int_{0}^{\pi}\cos\left(  \left(  \ell_{1}+\ell_{2}\right)
\alpha_{3}-\ell_{3}\beta_{2}\right)  \frac{\cos\left(  \ell_{1}\alpha_{2}%
-\ell_{2}\alpha_{1}\right)  }{\rho_{2}\kappa\sin\alpha_{1}}\chi\left(
\bigtriangleup|\rho_{1},\rho_{2},\kappa\right)  d\beta_{2},
\end{multline*}
where the notations correspond to the Figure \ref{Quadrilateral}, and
$\chi\left(  \bigtriangleup|\rho_{1},\rho_{2},\kappa\right)  $ is zero if
$\left(  \rho_{1},\rho_{2},\kappa\right)  $ does not form a triangle,
otherwise it is $1$.
\end{lemma}

\begin{proof}
The equation (\ref{equ_Bess2}) and Lemma \ref{Lemma_Jprod3} give
\begin{align*}
J_{\ell_{3}}\left(  \rho_{3}\lambda\right)  J_{\ell_{1}+\ell_{2}+\ell_{3}%
}\left(  \rho_{4}\lambda\right)   &  =\frac{1}{\pi}\int_{\left\vert \rho
_{4}-\rho_{3}\right\vert }^{\rho_{3}+\rho_{4}}J_{\ell_{1}+\ell_{2}}\left(
\kappa\lambda\right)  \frac{\cos\left(  \left(  \ell_{1}+\ell_{2}\right)
\alpha_{3}-\ell_{3}\beta_{2}\right)  \kappa d\kappa}{\rho_{3}\rho_{4}\sin
\beta_{2}},\\
\int_{0}^{\infty}J_{\ell_{1}}\left(  \rho_{1}\lambda\right)  J_{\ell_{2}%
}\left(  \rho_{2}\lambda\right)  J_{\ell_{1}+\ell_{2}}\left(  \kappa
\lambda\right)  \lambda d\lambda &  =\frac{\cos\left(  \ell_{1}\alpha_{2}%
-\ell_{2}\alpha_{1}\right)  }{\pi\rho_{1}\kappa\sin\alpha_{2}}\chi\left(
\bigtriangleup|\rho_{1},\rho_{2},\kappa\right)  ,
\end{align*}
hence
\begin{align*}
&  \int_{0}^{\infty}J_{\ell_{1}}\left(  \rho_{1}\lambda\right)  J_{\ell_{2}%
}\left(  \rho_{2}\lambda\right)  J_{\ell_{3}}\left(  \rho_{3}\lambda\right)
J_{\ell_{1}+\ell_{2}+\ell_{3}}\left(  \rho_{4}\lambda\right)  \lambda
d\lambda\\
&  =\frac{1}{\pi}\int_{\left\vert \rho_{4}-\rho_{3}\right\vert }^{\rho
_{4}+\rho_{3}}\frac{\cos\left(  \left(  \ell_{1}+\ell_{2}\right)  \alpha
_{3}-\ell_{3}\beta_{2}\right)  }{\rho_{3}\rho_{4}\sin\beta_{2}}\int
_{0}^{\infty}J_{\ell_{1}}\left(  \rho_{1}\lambda\right)  J_{\ell_{2}}\left(
\rho_{2}\lambda\right)  J_{\ell_{1}+\ell_{2}}\left(  \kappa\lambda\right)
\lambda d\lambda\kappa d\kappa\\
&  =\frac{1}{\pi^{2}}\int_{\left\vert \rho_{4}-\rho_{3}\right\vert }^{\rho
_{4}+\rho_{3}}\frac{\cos\left(  \left(  \ell_{1}+\ell_{2}\right)  \alpha
_{3}-\ell_{3}\beta_{2}\right)  }{\rho_{3}\rho_{4}\sin\beta_{2}}\frac
{\cos\left(  \ell_{1}\alpha_{2}-\ell_{2}\alpha_{1}\right)  }{\rho_{1}%
\kappa\sin\alpha_{2}}\chi\left(  \bigtriangleup|\rho_{1},\rho_{2}%
,\kappa\right)  \kappa d\kappa\\
&  =\frac{1}{\pi^{2}}\int_{0}^{\pi}\cos\left(  \left(  \ell_{1}+\ell
_{2}\right)  \alpha_{3}-\ell_{3}\beta_{2}\right)  \frac{\cos\left(  \ell
_{1}\alpha_{2}-\ell_{2}\alpha_{1}\right)  }{\rho_{2}\kappa\sin\alpha_{1}}%
\chi\left(  \bigtriangleup|\rho_{1},\rho_{2},\kappa\right)  d\beta_{2},
\end{align*}
where $\sqrt{\left(  2\rho_{3}\rho_{4}\right)  ^{2}-\left(  \kappa^{2}%
-\rho_{3}^{2}-\rho_{4}^{2}\right)  ^{2}}=2\rho_{3}\rho_{4}\sin\beta_{2}$,
$\kappa d\kappa=\rho_{3}\rho_{4}\sin\left(  \beta_{2}\right)  d\beta_{2}$,
$\rho_{1}\sin\beta_{1}=\kappa\sin\alpha_{1}$, see Figure \ref{Quadrilateral}.
Note that if we are given wave numbers $\left(  \rho_{1},\rho_{2},\rho
_{3},\rho_{4}\right)  $ and if $\kappa$ changes then not only $\beta_{2}$ will
change but all the angles as well.
\end{proof}

For further generalization of Lemma \ref{Lemma_Jprod3} we consider
multilaterals on the plane. A multilateral of order $5$, say, has $5$ vertices
and $2$ diagonals, see Figure \ref{multilateral}. Invariants under the motion
of a rigid body are the angles the length of the sides and diagonals. The
multilateral will be well defined if the length of the sides and diagonals are
given, one may replace the diagonals by the angle opposite them. For instance
the $\kappa_{2}=\left\vert \underline{\kappa}_{2}\right\vert $ and angle
$\beta_{2,2}$ are equivalent in determining the triangle together with sides
$\rho_{4}=\left\vert \underline{\omega}_{4}\right\vert $ and $\rho
_{5}=\left\vert \underline{\omega}_{5}\right\vert $.
\begin{figure}
[ptbh]
\begin{center}
\includegraphics[
height=3.4454in,
width=6.3218in
]%
{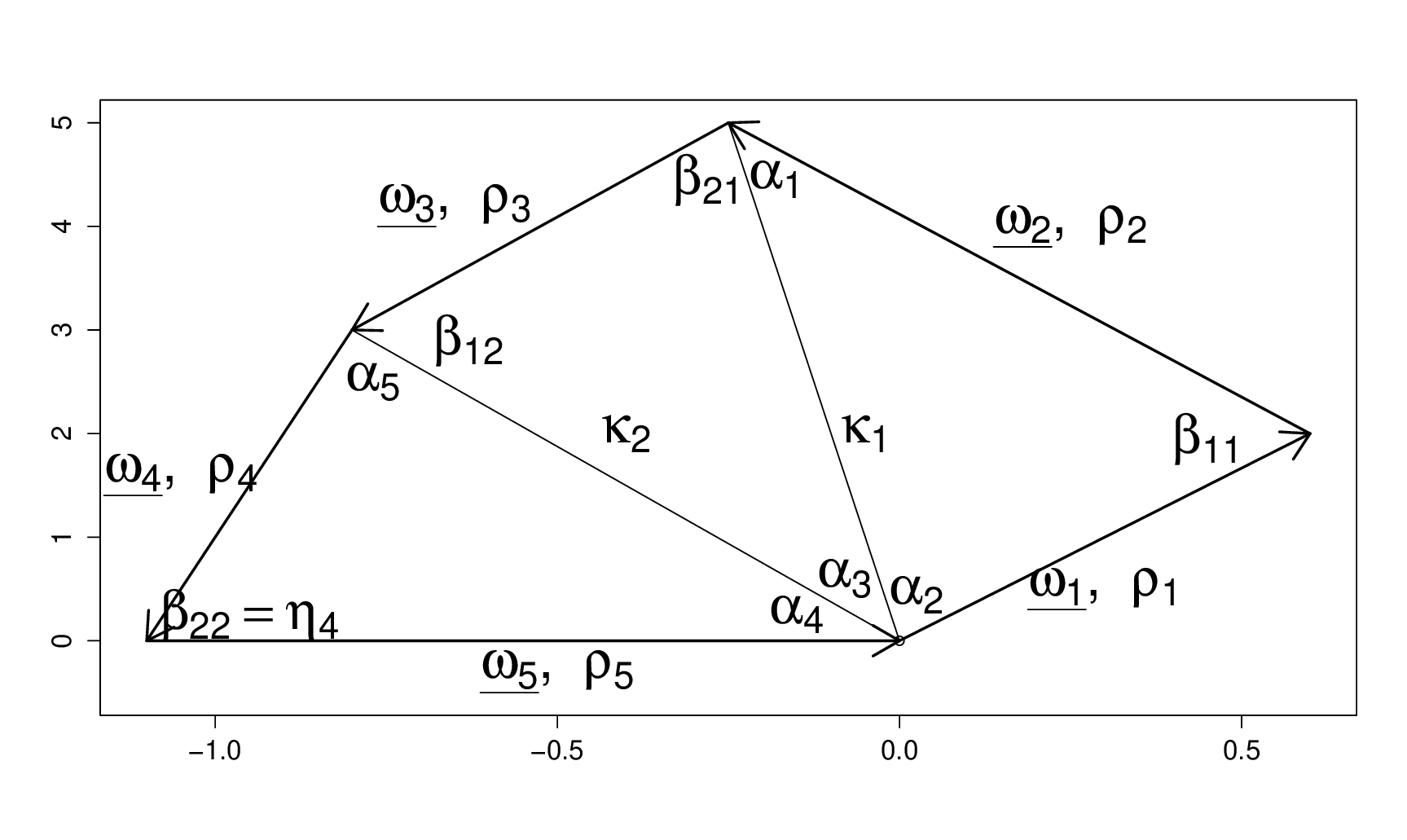}%
\caption{Multilateral}%
\label{multilateral}%
\end{center}
\end{figure}

\begin{theorem}
\label{Theorem_Jprod} Let $p\geq4$ and consider a multilateral of order $p$
then
\begin{multline*}
\int_{0}^{\infty}J_{\Sigma_{1}^{p-1}\ell_{k}}\left(  \rho_{p}\lambda\right)
\prod\limits_{k=1}^{p-1}J_{\ell_{k}}\left(  \rho_{k}\lambda\right)  \lambda
d\lambda\\
=\frac{1}{\pi^{p-2}}\int_{0}^{\pi}\cdots\int_{0}^{\pi}\frac{\cos\left(
\ell_{1}\alpha_{2}-\ell_{2}\alpha_{1}\right)  }{\rho_{2}\kappa_{1}\sin\left(
\alpha_{1}\right)  }\prod\limits_{k=2}^{p-2}\cos\left(  \alpha_{k+1}%
{\textstyle\sum_{j=1}^{k}}
\ell_{j}-\ell_{k+1}\beta_{k-1,2}\right)  \chi\left(  \bigtriangleup|\rho
_{k+2},\kappa_{k},\kappa_{k+1}\right)  d\beta_{k-1,2},
\end{multline*}
where each angle $\alpha_{k}$ is opposite to $\rho_{k}$ and angles
$\beta_{k,1}$, $\beta_{k,2}$ are opposite to diagonal $\kappa_{k}$ on the
right and on the left respectively, see Figure \ref{multilateral} for notations.
\end{theorem}

\begin{proof}
A multilateral can be split up $p-2$ triangles, see Figure \ref{multilateral}.
We show that from $p=4$ follows $p=5$, such that the pattern of general
induction shows up. By the Addition Theorem we have
\[
J_{\ell_{4}}\left(  \rho_{4}\lambda\right)  J_{\ell_{1}+\ell_{2}+\ell_{3}%
+\ell_{4}}\left(  \rho_{5}\lambda\right)  =\frac{1}{\pi}\int_{0}^{\pi}%
J_{\ell_{1}+\ell_{2}+\ell_{3}}\left(  \kappa_{2}\lambda\right)  \cos\left(
\left(  \ell_{1}+\ell_{2}+\ell_{3}\right)  \alpha_{4}-\ell_{4}\beta
_{2,2}\right)  d\beta_{2,2},
\]
and the result of Lemma \ref{Lemma_Jprod4} leads us to the formula
\begin{multline*}
\int_{0}^{\infty}J_{\ell_{1}}\left(  \rho_{1}\lambda\right)  J_{\ell_{2}%
}\left(  \rho_{2}\lambda\right)  J_{\ell_{3}}\left(  \rho_{3}\lambda\right)
J_{\ell_{1}+\ell_{2}+\ell_{3}}\left(  \kappa_{2}\lambda\right)  \lambda
d\lambda\\
=\frac{1}{\pi^{2}}\int_{0}^{\pi}\cos\left(  \left(  \ell_{1}+\ell_{2}\right)
\alpha_{3}-\ell_{3}\beta_{1,2}\right)  \cos\left(  \ell_{1}\alpha_{2}-\ell
_{2}\alpha_{1}\right)  \frac{d\beta_{1,2}}{\rho_{2}\kappa_{1}\sin\alpha_{1}}%
\end{multline*}
hence we obtain%
\begin{align*}
&  \int_{0}^{\infty}J_{\ell_{1}}\left(  \rho_{1}\lambda\right)  J_{\ell_{2}%
}\left(  \rho_{2}\lambda\right)  J_{\ell_{3}}\left(  \rho_{3}\lambda\right)
J_{\ell_{4}}\left(  \rho_{4}\lambda\right)  J_{\ell_{1}+\ell_{2}+\ell_{3}%
+\ell_{4}}\left(  \rho_{5}\lambda\right)  \lambda d\lambda\\
&  =\frac{1}{\pi^{3}}\int_{0}^{\infty}J_{\ell_{1}}\left(  \rho_{1}%
\lambda\right)  J_{\ell_{2}}\left(  \rho_{2}\lambda\right)  J_{\ell_{3}%
}\left(  \rho_{3}\lambda\right)  \int_{0}^{\pi}J_{\ell_{1}+\ell_{2}+\ell_{3}%
}\left(  \kappa_{2}\lambda\right)  \cos\left(  \left(  \ell_{1}+\ell_{2}%
+\ell_{3}\right)  \alpha_{4}-\ell_{4}\beta_{2,2}\right)  d\beta_{2,2}\lambda
d\lambda\\
&  =\frac{1}{\pi^{3}}\int_{0}^{\pi}\int_{0}^{\pi}\cos\left(  \ell_{1}%
\alpha_{2}-\ell_{2}\alpha_{1}\right)  \cos\left(  \left(  \ell_{1}+\ell
_{2}\right)  \alpha_{3}-\ell_{3}\beta_{1,2}\right)  \frac{\cos\left(  \left(
\ell_{1}+\ell_{2}+\ell_{3}\right)  \alpha_{4}-\ell_{4}\beta_{2,2}\right)
d\beta_{2,2}d\beta_{1,2}}{\rho_{2}\kappa_{1}\sin\alpha_{1}}.
\end{align*}

\end{proof}

\section{Dirac-function in polar coordinates \label{Proof_Dirac4a}}

The covariance function of a homogenous field has a symmetric spectral
representation where the Dirac 'function' is involved. Let $\delta\left(
\cdot\right)  $ denote the Dirac 'function', more precisely $\delta\left(
\cdot\right)  $ is a distribution putting all the mass at zero, for instance
the integral of Bessel functions provides Dirac function
\begin{equation}
\int\limits_{0}^{\infty}J_{\ell}\left(  \rho r\right)  J_{\ell}\left(  \kappa
r\right)  rdr=\frac{\delta\left(  \rho-\kappa\right)  }{\rho},
\label{Dirac_Bessel}%
\end{equation}
see \cite{Arfken2001} Sect 11. p691. We shall apply the Jacobi-Anger expansion
on the plane
\begin{equation}
e^{i\rho r\cos\left(  \varphi-\eta\right)  }=\sum_{\ell=-\infty}^{\infty
}i^{\ell}J_{\ell}\left(  \rho r\right)  e^{i\ell\left(  \varphi-\eta\right)
}. \label{Expans_Jacobi_Anger}%
\end{equation}

In order to understand the influence of the Dirac 'function' in polar
coordinates we express it by the integral through the Jacobi-Anger expansion
(\ref{Expans_Jacobi_Anger}) and obtain%
\begin{equation}
\delta\left(  \Sigma_{1}^{p}\rho_{k}\underline{\widehat{\omega}}_{k}\right)
=\frac{1}{\left(  2\pi\right)  ^{2}}\int\limits_{\mathbb{R}^{2}}e^{i\left(
\underline{\lambda}\cdot\Sigma_{1}^{p}\underline{\omega}_{k}\right)
}d\underline{\lambda}, \label{Dirac2D}%
\end{equation}
where the sum of vectors is invariant under permutation.
\begin{align*}
\delta\left(  \Sigma_{1}^{p}\rho_{k}\underline{\widehat{\omega}}_{k}\right)
&  =\frac{1}{\left(  2\pi\right)  ^{2}}\int_{0}^{\infty}\int_{0}^{2\pi}%
{\displaystyle\prod\limits_{k=1}^{p}}
\sum_{m_{k}=-\infty}^{\infty}i^{m_{k}}J_{m_{k}}\left(  \rho_{k}\lambda\right)
e^{im_{k}\left(  \eta_{k}-\xi\right)  }\lambda d\lambda d\xi\\
&  =\frac{1}{\left(  2\pi\right)  ^{2}}\int_{0}^{\infty}\int_{0}^{2\pi}%
\sum_{m_{1:p}=-\infty}^{\infty}i^{\Sigma_{1}^{p}m_{k}}e^{i\Sigma_{1}^{p}%
m_{k}\left(  \eta_{k}-\xi\right)  }%
{\displaystyle\prod\limits_{k=1}^{p}}
J_{m_{k}}\left(  \rho_{k}\lambda\right)  \lambda d\lambda d\xi\\
&  =\frac{\delta_{\Sigma_{1}^{p}m_{k}}}{2\pi}\int_{0}^{\infty}\sum
_{m_{1:p-1}=-\infty}^{\infty}e^{i\Sigma_{1}^{p-1}m_{k}\left(  \eta_{k}%
-n_{p}\right)  }J_{-\Sigma_{1}^{p-1}m_{k}}\left(  \rho_{p}\lambda\right)
{\displaystyle\prod\limits_{k=1}^{p-1}}
J_{m_{k}}\left(  \rho_{k}\lambda\right)  \lambda d\lambda\\
&  =\frac{\delta_{\Sigma_{1}^{p}m_{k}}}{2\pi}\int_{0}^{\infty}\sum
_{m_{1:p-1}=-\infty}^{\infty}\left(  -1\right)  ^{\Sigma_{1}^{p-1}m_{k}%
}e^{i\Sigma_{1}^{p-1}m_{k}\left(  \eta_{k}-n_{p}\right)  }J_{\Sigma_{1}%
^{p-1}m_{k}}\left(  \rho_{p}\lambda\right)
{\displaystyle\prod\limits_{k=1}^{p-1}}
J_{m_{k}}\left(  \rho_{k}\lambda\right)  \lambda d\lambda\\
&  =\frac{\delta_{\Sigma_{1}^{p}m_{k}}}{2\pi}\int_{0}^{\infty}\sum
_{m_{1:p-1}=-\infty}^{\infty}e^{i\Sigma_{1}^{p-1}m_{k}\left(  \eta_{k}%
-n_{p}-\pi\right)  }J_{\Sigma_{1}^{p-1}m_{k}}\left(  \rho_{p}\lambda\right)
{\displaystyle\prod\limits_{k=1}^{p-1}}
J_{m_{k}}\left(  \rho_{k}\lambda\right)  \lambda d\lambda
\end{align*}
since $\Sigma_{1}^{p}m_{k}=0$, $m_{p}=$ $-\Sigma_{1}^{p-1}m_{k}$. We can apply
here the Theorem \ref{Theorem_Jprod} for a clear expression. Some particular
cases as follows.

\begin{enumerate}
\item If $\boldsymbol{p=2}$,%
\begin{align*}
\delta\left(  \Sigma_{1}^{2}\rho_{k}\underline{\widehat{\omega}}_{k}\right)
&  =\frac{1}{2\pi}\int_{0}^{\infty}\sum_{m=-\infty}^{\infty}\left(  -1\right)
^{m}e^{im\left(  \eta_{1}-\eta_{2}\right)  }J_{m}\left(  \rho_{1}%
\lambda\right)  J_{m}\left(  \rho_{2}\lambda\right)  \lambda d\lambda\\
&  =\frac{1}{2\pi}\frac{\delta\left(  \rho_{1}-\rho_{2}\right)  }{\rho_{1}%
}\sum_{m=-\infty}^{\infty}e^{im\left(  \eta_{1}-\eta_{2}-\pi\right)  }\\
&  =\frac{\delta\left(  \rho_{1}-\rho_{2}\right)  }{\rho_{2}}\delta\left(
\eta_{1}-\eta_{2}-\pi\right)  ,
\end{align*}
hence the integral is taken according to the subspace $\rho_{1}=\rho_{2}$ and
$\eta_{1}=\eta_{2}+\pi$, it corresponds to $\underline{\omega}_{1}%
=-\underline{\omega}_{2}$, this subspace is the one what is expected.

\item For $\boldsymbol{p=3}$, we apply Lemma \ref{Lemma_Jprod3},
\begin{align}
&  \delta\left(  \Sigma_{1}^{3}\rho_{k}\underline{\widehat{\omega}}%
_{k}\right)  =\frac{1}{2\pi}\int_{0}^{\infty}\sum_{m_{1:2}=-\infty}^{\infty
}e^{i\Sigma_{1}^{2}m_{k}\left(  \eta_{k}-\eta_{3}-\pi\right)  }J_{m_{1}+m_{2}%
}\left(  \rho_{3}\lambda\right)
{\displaystyle\prod\limits_{k=1}^{2}}
J_{m_{k}}\left(  \rho_{k}\lambda\right)  \lambda d\lambda\label{Dirac3}\\
&  =\frac{\chi\left(  \triangle|\rho_{1},\rho_{2},\rho_{3}\right)  }{2\pi
^{2}\rho_{2}\rho_{3}\sin\alpha_{1}}\sum_{m_{1:2}=-\infty}^{\infty}%
e^{i\Sigma_{1}^{2}m_{k}\left(  \eta_{k}-\eta_{3}-\pi\right)  }\cos\left(
m_{2}\alpha_{1}-m_{1}\alpha_{2}\right) \nonumber\\
=\frac{\chi\left(  \triangle|\rho_{1},\rho_{2},\rho_{3}\right)  }{\left(
2\pi\right)  ^{2}\rho_{2}\rho_{3}\sin\alpha_{1}}  &  \left(  \sum
_{m_{1}=-\infty}^{\infty}e^{im_{1}\left(  \eta_{1}-\eta_{3}-\pi-\alpha
_{2}\right)  }\sum_{m_{2}=-\infty}^{\infty}e^{im_{2}\left(  \eta_{2}-\eta
_{3}-\pi+\alpha_{1}\right)  }\right. \nonumber\\
&  \left.  +\sum_{m_{1}=-\infty}^{\infty}e^{im_{1}\left(  \eta_{1}-\eta
_{3}-\pi+\alpha_{2}\right)  }\sum_{m_{2}=-\infty}^{\infty}e^{im_{2}\left(
\eta_{2}-\eta_{3}-\pi-\alpha_{1}\right)  }\right) \nonumber\\
&  =\frac{\chi\left(  \triangle|\rho_{1},\rho_{2},\rho_{3}\right)  }{\rho
_{2}\rho_{3}\sin\alpha_{1}}\left(  \delta\left(  \eta_{1}-\eta_{3}-\pi
-\alpha_{2}\right)  \delta\left(  \eta_{2}-\eta_{3}-\pi+\alpha_{1}\right)
\right. \nonumber\\
&  \left.  +\delta\left(  \eta_{1}-\eta_{3}-\pi+\alpha_{2}\right)
\delta\left(  \eta_{2}-\eta_{3}-\pi-\alpha_{1}\right)  \right)  ,\nonumber
\end{align}
where the notations of Figure \ref{Triangle} are used. Here the Dirac
'function' is concentrated on the subspace when $\left(  \rho_{1},\rho
_{2},\rho_{3}\right)  $ forms a triangle, this triangle defines angles
$\alpha_{1}$, $\alpha_{2}$, $\alpha_{3}$, see Figure \ref{Triangle}. Once
$\alpha_{1}$, $\alpha_{2}$, $\alpha_{3}$, are given there are two possible
choices for angles $\eta_{1}-\eta_{3}$, $\eta_{2}-\eta_{3}$, such that
$\eta_{3}$ varies from $0$ to $2\pi$. Actually we plotted the case when
$\eta_{3}=0$, see Figure \ref{Triangle}. One can also check that the set
$\Sigma_{1}^{3}\rho_{k}\underline{\widehat{\omega}}_{k}=0$, will not change if
we put $m_{2}=-m_{1}-m_{3}$in (\ref{Dirac3}) instead of $m_{3}=-m_{1}-m_{2}$,
although it may be counted when the principal domain of the bispectrum is of interest.

\item Similarly, for\textbf{ }$\boldsymbol{p=4}$, we have
\begin{align}
&  \delta\left(  \Sigma_{1}^{4}\rho_{k}\underline{\widehat{\omega}}_{k}\right)
\label{Dirac4}\\
&  =\frac{1}{2\pi}\int_{0}^{\infty}\sum_{m_{1:3}=-\infty}^{\infty}%
e^{i\Sigma_{1}^{3}m_{k}\left(  \eta_{k}-\eta_{4}-\pi\right)  }J_{m_{1}%
+m_{2}+m_{3}}\left(  \rho_{4}\lambda\right)
{\displaystyle\prod\limits_{k=1}^{3}}
J_{m_{k}}\left(  \rho_{k}\lambda\right)  \lambda d\lambda\nonumber\\
&  =\frac{1}{2\pi^{3}}\sum_{m_{1:3}=-\infty}^{\infty}e^{i\Sigma_{1}^{3}%
m_{k}\left(  \eta_{k}-\eta_{4}-\pi\right)  }\nonumber\\
&  \int_{0}^{\pi}\cos\left(  m_{1}\alpha_{2}-m_{2}\alpha_{1}\right)
\cos\left(  \left(  m_{1}+m_{2}\right)  \alpha_{3}-m_{3}\beta_{2}\right)
\frac{\chi\left(  \bigtriangleup|\rho_{1},\rho_{2},\kappa\right)  }{\rho
_{2}\kappa\sin\alpha_{1}}d\beta_{2},\nonumber
\end{align}
see Lemma \ref{Lemma_Jprod4} and Figure \ref{Quadrilateral} for this case.
Since $\rho_{k}\underline{\widehat{\omega}}_{k}$ are given one can expect some
more precise expression. Indeed
\begin{align}
&  \frac{1}{2\pi^{3}}\sum_{m_{1:3}=-\infty}^{\infty}e^{i\Sigma_{1}^{3}%
m_{k}\left(  \eta_{k}-\eta_{4}-\pi\right)  }\cos\left(  m_{1}\alpha_{2}%
-m_{2}\alpha_{1}\right)  \cos\left(  \left(  m_{1}+m_{2}\right)  \alpha
_{3}-m_{3}\beta_{2}\right) \label{Dirac4_Prod}\\
&  =\delta\left(  \eta_{1}-\eta_{4}-\pi+\alpha_{2}+\alpha_{3}\right)
\delta\left(  \eta_{2}-\eta_{4}-\pi-\alpha_{1}+\alpha_{3}\right)
\delta\left(  \eta_{3}-\eta_{4}-\pi-\beta_{2}\right) \nonumber\\
&  +\delta\left(  \eta_{1}-\eta_{4}-\pi+\alpha_{2}-\alpha_{3}\right)
\delta\left(  \eta_{2}-\eta_{4}-\pi-\alpha_{1}-\alpha_{3}\right)
\delta\left(  \eta_{3}-\eta_{4}-\pi+\beta_{2}\right) \nonumber\\
&  +\delta\left(  \eta_{1}-\eta_{4}-\pi-\alpha_{2}+\alpha_{3}\right)
\delta\left(  \eta_{2}-\eta_{4}-\pi+\alpha_{1}+\alpha_{3}\right)
\delta\left(  \eta_{3}-\eta_{4}-\pi-\beta_{2}\right) \nonumber\\
&  +\delta\left(  \eta_{1}-\eta_{4}-\pi-\alpha_{2}-\alpha_{3}\right)
\delta\left(  \eta_{2}-\eta_{4}-\pi+\alpha_{1}-\alpha_{3}\right)
\delta\left(  \eta_{3}-\eta_{4}-\pi+\beta_{2}\right)  .\nonumber
\end{align}
Now, for a given $\alpha_{1},\rho_{2},\rho_{3},\rho_{4}$, the diagonal
$\kappa$ and $\beta_{2}$, are equivalent, \newline$\kappa\left(  \beta
_{2}\right)  =\sqrt{\rho_{3}^{2}+\rho_{4}^{2}-2\rho_{3}\rho_{4}\cos\beta_{2}}%
$, say, let $\beta_{2}$ be the subject of changes. Hence $\alpha_{3}$ is
determined, together with $\alpha_{1}$ and $\alpha_{2}$, see Figure
\ref{Quadrilateral}. It follows that $\eta_{3}-\eta_{4}=\pi\pm\beta_{2}$. each
choice of $\eta_{3}-\eta_{4}$ we have two possibilities $\eta_{1}-\eta_{4}$
and $\eta_{2}-\eta_{4}$, are defined as a function of $\alpha_{1}$,
$\alpha_{2}$ and $\alpha_{3}$.
\end{enumerate}

\section{Cumulants of spectral measures $Z_{\ell}\left(  \rho d\rho\right)
$\label{Proof_cumulants}}

We generalize the joint cumulant stochastic spectral measures
\begin{multline*}
\operatorname*{Cum}\left(  Z_{0}\left(  \rho_{1}d\rho_{1}\right)  ,Z_{\ell
}\left(  \rho_{2}d\rho_{2}\right)  ,Z_{-\ell}\left(  \rho_{3}d\rho_{3}\right)
\right) \\
=2\left(  -1\right)  ^{\ell}\chi\left(  \bigtriangleup|\rho_{1},\rho_{2}%
,\rho_{3}\right)  \frac{\cos\left(  \ell\arccos\left(  R\right)  \right)
}{\rho_{2}\rho_{3}\sqrt{1-R^{2}}}S_{3}\left(  \rho_{1},\rho_{2},\rho
_{3}\right)
{\textstyle\prod\limits_{k=1}^{3}}
\rho_{k}d\rho_{k},
\end{multline*}
where $R=\left(  \rho_{2}^{2}+\rho_{3}^{2}-\rho_{1}^{2}\right)  /\left(
2\rho_{2}\rho_{3}\right)  =\cos\alpha_{1}$ and $\chi\left(  \bigtriangleup
|\rho_{1},\rho_{2},\rho_{3}\right)  =1$, if $\rho_{1},\rho_{2},\rho_{3}%
\ $constitute a triangle, $0$ otherwise, see \cite{Terdik2014Publi} in order
to get the formula for trispectrum and higher order spectra. $\chi\left(
\bigtriangleup|\rho_{1},\rho_{2},\rho_{3}\right)  $ implies that the wave
numbers $\rho_{1}$, $\rho_{2}$, and $\rho_{3}$ should satisfy the triangle relation.

Consider the fourth order cumulant
\begin{multline*}
\operatorname*{Cum}\left(  Z_{0}\left(  \rho_{1}d\rho_{1}\right)  ,Z_{\ell
_{2}}\left(  \rho_{2}d\rho_{2}\right)  ,Z_{\ell_{3}}\left(  \rho_{3}d\rho
_{3}\right)  ,Z_{-\left(  \ell_{2}+\ell_{3}\right)  }\left(  \rho_{4}d\rho
_{4}\right)  \right) \\
=\iiiint\limits_{0}^{2\pi}e^{-i\ell_{2}\left(  \eta_{2}-\eta_{4}\right)
-i\ell_{3}\left(  \eta_{3}-\eta_{4}\right)  }\delta\left(  \Sigma_{1}^{4}%
\rho_{k}\underline{\widehat{\omega}}_{k}\right)  S_{4}\left(  \alpha_{1}%
,\rho_{2:4},\beta_{2}\right)
{\textstyle\prod\limits_{k=1}^{4}}
d\eta_{k}%
{\textstyle\prod\limits_{k=1}^{4}}
\rho_{k}d\rho_{k},
\end{multline*}
replace the Dirac-function by (\ref{Dirac2D}), (\ref{Dirac4}), and use the
orthogonality of the 'spherical harmonics',
\[
\iiiint\limits_{0}^{2\pi}e^{-i\ell_{2}\left(  \eta_{2}-\eta_{4}\right)
-i\ell_{3}\left(  \eta_{3}-\eta_{4}\right)  \ }e^{i\Sigma_{1}^{3}m_{k}\left(
\eta_{k}-\eta_{4}\right)  }%
{\textstyle\prod\limits_{k=1}^{4}}
d\eta_{k}=\delta_{m_{1}}\delta_{m_{2}-\ell_{2}}\delta_{m_{3}-\ell_{3}}\left(
2\pi\right)  ^{4},
\]
and a particular case of Lemma \ref{Lemma_Jprod4}%
\begin{multline*}
\int_{0}^{\infty}J_{0}\left(  \rho_{1}\lambda\right)  J_{\ell_{2}}\left(
\rho_{2}\lambda\right)  J_{\ell_{3}}\left(  \rho_{3}\lambda\right)
J_{\ell_{2}+\ell_{3}}\left(  \rho_{4}\lambda\right)  \lambda d\lambda\\
=\frac{1}{\pi^{2}}\int_{0}^{\pi}\cos\left(  \left(  \ell_{1}+\ell_{2}\right)
\alpha_{3}-\ell_{3}\beta_{2}\right)  \frac{\cos\left(  \ell_{1}\alpha_{2}%
-\ell_{2}\alpha_{1}\right)  }{\pi\rho_{2}\kappa\sin\alpha_{1}}\chi\left(
\bigtriangleup|\rho_{1},\rho_{2},\kappa\right)  d\beta_{2},
\end{multline*}
see Figure \ref{Quadrilateral} for notations. The result is
\begin{multline*}
\operatorname*{Cum}\left(  Z_{0}\left(  \rho_{1}d\rho_{1}\right)  ,Z_{\ell
_{2}}\left(  \rho_{2}d\rho_{2}\right)  ,Z_{\ell_{3}}\left(  \rho_{3}d\rho
_{3}\right)  ,Z_{-\left(  \ell_{2}+\ell_{3}\right)  }\left(  \rho_{4}d\rho
_{4}\right)  \right) \\
=4\left(  -1\right)  ^{\ell_{2}+\ell_{3}}\int_{0}^{\pi}\frac{\chi\left(
\bigtriangleup|\rho_{1},\rho_{2},\kappa\right)  }{\rho_{1}\kappa\sin\alpha
_{2}}\cos\left(  \ell_{2}\alpha_{1}\right)  \cos\left(  \ell_{2}\alpha
_{3}-\ell_{3}\beta_{2}\right)  S_{4}\left(  \alpha_{1},\rho_{2:4},\beta
_{2}\right)  d\beta_{2}%
{\textstyle\prod\limits_{k=1}^{4}}
\rho_{k}d\rho_{k}.
\end{multline*}

We obtain the cumulant similarly for general $p$, it follows from a particular
case of Theorem \ref{Theorem_Jprod}, when $\ell_{1}=0$, see Figure
\ref{multilateral}.

\begin{lemma}
\label{Lemma_genCum_Z}
\begin{multline*}
\operatorname*{Cum}\left(  Z_{0}\left(  \rho_{1}d\rho_{1}\right)  ,Z_{\ell
_{2}}\left(  \rho_{2}d\rho_{2}\right)  ,Z_{\ell_{3}}\left(  \rho_{3}d\rho
_{3}\right)  ,\ldots,Z_{-\left(  \ell_{2}+\ell_{3}\cdots+\ell_{p-1}\right)
}\left(  \rho_{p}d\rho_{p}\right)  \right) \\
=\left(  -1\right)  ^{\ell_{2}+\ell_{3}\cdots+\ell_{p-1}}2^{p-2}\int_{0}^{\pi
}\cdots\int_{0}^{\pi}S_{p}\left(  \alpha_{1},\rho_{2:p},\beta_{1:p-3,2}%
\right)  L\left(  \ell_{2:p-1},\alpha_{1},\beta_{1:p-3,1}\right)
{\textstyle\prod\limits_{k=2}^{p-2}}
d\beta_{k-1,2}%
{\textstyle\prod\limits_{m=1}^{p}}
\rho_{m}d\rho_{m},
\end{multline*}
where
\[
L\left(  \ell_{2:p-1},\alpha_{1:p-1},\beta_{1:p-1,2}\right)  =\frac
{\cos\left(  \ell_{2}\alpha_{1}\right)  }{\rho_{2}\kappa_{1}\sin\left(
\alpha_{1}\right)  }\prod\limits_{k=2}^{p-2}\cos\left(  \alpha_{k+1}%
{\textstyle\sum_{j=2}^{k}}
\ell_{j}-\ell_{k+1}\beta_{k-1,2}\right)  \chi\left(  \bigtriangleup|\rho
_{k+2},\kappa_{k},\kappa_{k+1}\right)  .
\]

\end{lemma}

\bibliographystyle{aalpha}
\bibliography{00BiblMM13}

\end{document}